\documentclass{amsart}
 \usepackage{graphicx}
\usepackage{amssymb,amsmath}

\newtheorem{theorem}{Theorem}[section]
\newtheorem{lemma}{Lemma}[section]
\newtheorem{corollary}{Corollary}[section]
\newtheorem{conjecture}{Conjecture}[section]

\numberwithin{equation}{section}

\newcommand{\cE}{{\mathcal{E}}}
\newcommand{\cF}{{\mathcal{F}}}

\newcommand{\cK}{{\mathcal{K}}}

\newcommand{\cP}{{\mathcal{P}}}

\newcommand{\cS}{{\mathcal{S}}}

        \newcommand{\field}[1]{{\mathbb{#1}}}
        
        \newcommand{\ZZ}{\field{Z}}
        
        \newcommand{\RR}{\field{R}}
        \newcommand{\CC}{\field{C}}

\newcommand{\Dom}{\mbox{\rm Dom}}

\newcommand{\Tr}{\mbox{\rm Tr}}

\allowdisplaybreaks

\begin{document}

\title[Periodic Schr\"odinger operators with hypersurface magnetic wells]{Spectral gaps for
periodic Schr\"odinger operators with hypersurface magnetic wells:
Analysis near the bottom}

\author{B. Helffer}

\address{D\'epartement de Math\'ematiques, B\^atiment 425, Univ
Paris-Sud et CNRS, F-91405 Orsay C\'edex, France}
\email{Bernard.Helffer@math.u-psud.fr}

\author{Y. A. Kordyukov }
\address{Institute of Mathematics, Russian Academy of Sciences, 112 Chernyshevsky
str. 450077 Ufa, Russia} \email{yurikor@matem.anrb.ru}

\thanks{Y.K. is partially supported by the Russian Foundation of Basic
Research (grant 06-01-00208).}

\begin{abstract}
We consider a periodic magnetic Schr\"odinger operator $H^h$,
depending on the semiclassical parameter $h>0$, on a noncompact
Riemannian manifold $M$ such that $H^1(M, \RR)=0$ endowed with a
properly discontinuous cocompact isometric action of a discrete
group. We assume that there is no electric field and that the
magnetic field has a periodic set of compact magnetic wells. We
suppose that the magnetic field vanishes regularly on a hypersurface
$S$. First, we prove upper and lower estimates for the bottom
$\lambda_0(H^h)$ of the spectrum of the operator $H^h$in $L^2(M)$.
Then, assuming the existence of non-degenerate miniwells for the
reduced spectral problem on $S$, we prove the existence of an
arbitrary large number of spectral gaps for the operator $H^h$ in
the region close to $\lambda_0(H^h)$, as $h\to 0$. In this case, we
also obtain upper estimates for the eigenvalues of the one-well
problem.
\end{abstract}
 \maketitle

\section{Preliminaries and main results}
Let $ M$ be a noncompact oriented manifold of dimension $n\geq 2$
equipped with a properly discontinuous action of a finitely
generated, discrete group $\Gamma$ such that $M/\Gamma$ is compact.
Suppose that $H^1(M, \RR) = 0$, i.e. any closed $1$-form on $M$ is
exact. Let $g$ be a $\Gamma$-invariant Riemannian metric and $\bf B$
a real-valued $\Gamma$-invariant closed 2-form on $M$. Assume that
$\bf B$ is exact and choose a real-valued 1-form $\bf A$ on $M$ such
that $d{\bf A} = \bf B$.

Thus, one has a natural mapping
\[
u\mapsto ih\,du+{\bf A}u
\]
from $C^\infty_c(M)$ to the space $\Omega^1_c(M)$ of smooth,
compactly supported one-forms on $M$. The Riemannian metric allows
to define scalar products in these spaces and consider the adjoint
operator
\[
(ih\,d+{\bf A})^* : \Omega^1_c(M)\to C^\infty_c(M).
\]
A Schr\"odinger operator with magnetic potential $\bf A$ is defined
by the formula
\[
H^h = (ih\,d+{\bf A})^* (ih\,d+{\bf A}).
\]
Here $h>0$ is a semiclassical parameter, which is assumed to be
small.

Choose local coordinates $X=(X_1,\ldots,X_n)$ on $M$. Write the
1-form $\bf A$ in the local coordinates as
\[
{\bf A}= \sum_{j=1}^nA_j(X)\,dX_j,
\]
the matrix of the Riemannian metric $g$ as
\[
g(X)=(g_{j\ell}(X))_{1\leq j,\ell\leq n}
\]
and its inverse as
\[
g(X)^{-1}=(g^{j\ell}(X))_{1\leq j,\ell\leq n}.
\]
Denote $|g(X)|=\det(g(X))$. Then the magnetic field $\bf B$ is given
by the following formula
\[
{\bf B}=\sum_{j<k}B_{jk}\,dX_j\wedge dX_k, \quad
B_{jk}=\frac{\partial A_k}{\partial X_j}-\frac{\partial
A_j}{\partial X_k}.
\]
Moreover, the operator $H^h$ has the form
\begin{multline*}
H^h=\frac{1}{\sqrt{|g(X)|}}\sum_{1\leq j,\ell\leq n}\left(i h
\frac{\partial}{\partial X_j}+A_j(X)\right)  \\
\times \left[\sqrt{|g(X)|} g^{j\ell}(X) \left(i h
\frac{\partial}{\partial X_\ell}+A_\ell(X)\right)\right].
\end{multline*}

For any $x\in M$, denote by $B(x)$ the anti-symmetric linear
operator on the tangent space $T_x{ M}$ associated with the 2-form
$\bf B$:
\[
g_x(B(x)u,v)={\bf B}_x(u,v),\quad u,v\in T_x{ M}.
\]
Recall that the intensity of the magnetic field is defined as
\[
{\Tr}^+ (B(x))=\sum_{\substack{\lambda_j(x)>0\\ i\lambda_j(x)\in
\sigma(B(x)) }}\lambda_j(x)=\frac{1}{2}\Tr([B^*(x)\cdot
B(x)]^{1/2}).
\]
It turns out that in many problems the function $x\mapsto h\cdot
{\Tr}^+ (B(x))$ can be considered as a magnetic potential, that is,
as a magnetic analog of the electric potential $V$ in a
Schr\"odinger operator $-h^2\Delta+V$.

We will also use the trace norm of $B(x)$:
\[
|B(x)|=[\Tr (B^*(x)\cdot B(x))]^{1/2}.
\]
It coincides with the norm of $B(x)$ with respect to the Riemannian
metric on the space of linear operators on $T_xM$ induced by the
Riemannian metric $g$ on $M$.

In this paper we will always assume that the magnetic field has a
periodic set of compact potential wells. More precisely, put
\[
b_0=\min \{{\Tr}^+ (B(x))\, :\, x\in { M}\}
\]
and assume that there exist a (connected) fundamental domain $\cF$
and a constant $\epsilon_0>0$ such that
\begin{equation}\label{YK:tr1}
  {\Tr}^+ (B(x)) \geq b_0+\epsilon_0, \quad x\in \partial\cF.
\end{equation}
For any $\epsilon_1 \leq \epsilon_0$, put
\[
U_{\epsilon_1} = \{x\in \cF\,:\, {\Tr}^+ (B(x)) < b_0+ \epsilon_1\}.
\]
Thus $U_{\epsilon_1}$ is an open subset of $\cF$ such that
$U_{\epsilon_1}\cap \partial\cF=\emptyset$ and, for $\epsilon_1 <
\epsilon_0$, $\overline{U_{\epsilon_1}}$ is compact and included in
the interior of $\cF$. Any connected component of $U_{\epsilon_1}$
with $\epsilon_1 < \epsilon_0$ and also any its translation under
the action of an element of $\Gamma$ can be understood as a magnetic
well. These magnetic wells are separated by potential barriers,
which are getting higher and higher when $h\to 0$ (in the
semiclassical limit).

We will consider the magnetic Schr\"odinger operator $H^h$ as an
unbounded self-adjoint operator in the Hilbert space $L^2(M)$ and
study gaps in the spectrum of this operator, which are located below
the top of potential barriers, that is, on the interval $[0,
h(b_0+\epsilon_0)]$. Here by a gap in the spectrum $\sigma(T)$ of a
self-adjoint operator $T$ in a Hilbert space we will mean any
connected component of the complement of $\sigma(T)$ in $\RR$, that
is, any maximal interval $(a,b)$ such that $(a,b)\cap \sigma(T) =
\emptyset $.

The problem of existence of gaps in the spectra of second order
periodic differential operators has been extensively studied
recently. Some related results on spectral gaps for periodic
magnetic Schr\"odinger operators can be found for example
in~\cite{BDP,gaps,diff2006,HS88,HSLNP345,HempelHerbst95,HempelPost02,HerbstNakamura,Ko04,bedlewo,KMS,MS,Nakamura95}
(see also the references therein).

In our case, the important role is played by the tunneling effect,
that is, by the possibility for the quantum particle described by
the Hamiltonian $H^h$ with such an energy to pass through a
potential barrier. Using the semiclassical analysis of the tunneling
effect, we showed in \cite{gaps} that the spectrum of the magnetic
Schr\"odinger operator $H^h$ on the interval is localized in an
exponentially small neighborhood of the spectrum of its Dirichlet
realization inside the wells. This result reduces the investigation
of some gaps in the spectrum of the operator $H^h$ to the study of
the eigenvalue distribution for a ``one-well'' operator and leads us
to suggest a general scheme of a proof of existence of spectral gaps
in \cite{diff2006}. We disregard in this paper the analysis of the
spectrum in the above mentioned exponentially small neighborhoods.

We consider the case when $b_0=0$ and the zero set of the magnetic
field has regular hypersurface pieces. More precisely, suppose that
there exists $x_0\in M$ such that $\mathbf B(x_0)=0$ and in a
neighborhood $U$ of $x_0$ the zero set of $\mathbf B$ is a smooth
oriented hypersurface $S$, and, moreover, there are constants $k\in
\ZZ$, $k>0$, and $C>0$ such that, for all $x\in U$, we have:
\begin{equation}\label{YK:B1}
C^{-1}d(x,S)^k\leq |B(x)|  \leq C d(x,S)^k\,.
\end{equation}

On compact manifolds, such a model was introduced for the first time
by Montgomery~\cite{Mont} and was further studied in
\cite{HM,Pan-Kwek,syrievienne,Qmath10}.

Denote by $N$ the external unit normal vector to $S$ and by
$\tilde{N}$ an arbitrary extension of $N$ to a smooth vector field
on $U$. Let $\omega_{0,1}$ be the smooth one form on $S$ defined,
for any vector field $V$ on $S$, by the formula
\[
\langle V,\omega_{0,1}\rangle(y)=\frac{1}{k!}\tilde{N}^k({\mathbf
B}(\tilde{N},\tilde{V}))(y), \quad y\in S,
\]
where $\tilde{V}$ is a $C^\infty$ extension of $V$ to $U$. By
\eqref{YK:B1}, it is easy to see that $\omega_{0,1}(x)\not=0$ for
any $x\in S$. Denote
\[
\omega_{\mathrm{min}}(B)=\inf_{x\in S} |\omega_{0,1}(x)|>0.
\]

For any $\alpha\in \RR$, denote by $\lambda_0(\alpha)$ the bottom of
the self-adjoint second order differential operator
\begin{equation}\label{e:Q0}
-\frac{d^2 }{d t^2}+ \left(\frac{t^{k+1}}{k+1}- \alpha \right)^2
\end{equation}
in $L^2(\RR)$. Put $\hat{\nu}:=\inf_{\alpha\in
\RR}\lambda_0(\alpha)$. It is clear that $\hat \nu \geq 0$. We
refer the reader to Section~\ref{s:model} for more properties.

In \cite{Qmath10}, we have proved the following result.

\begin{theorem}\label{YK:hypersurface}
For any $a$ and $b$ such that
\[
\hat{\nu}\, \omega_{\mathrm{min}}(B)^{\frac{2}{k+2}}<a < b
\]
and for any natural $N$, there exists $h_0>0$ such that, for any
$h\in (0,h_0]$, the spectrum of $H^h$ in the interval
\[
[h^{\frac{2k+2}{k+2}}a, h^{\frac{2k+2}{k+2}}b]
\]
has at least $N$ gaps.
\end{theorem}

In this paper we will concentrate our analysis on the region close
to the bottom of the spectrum. First, we state upper and lower
estimates for the bottom $\lambda_0(H^h)$ of the spectrum of the
operator $H^h$ in $L^2(M)$

\begin{theorem}\label{t:estimate}
Suppose that the operator $H^h$ satisfies condition \eqref{YK:tr1}
with some $\epsilon_0>0$, and that the zero set of the magnetic field
$\mathbf B$ is a smooth oriented hypersurface $S$. Moreover, assume
that there are some $k\in \ZZ, k>0$ and $C>0$ such that, for all $x$
in a neighborhood $U$ of $S$, we have:
\[
C^{-1}d(x,S)^k\leq |B(x)|  \leq C d(x,S)^k\,.
\]
Then there exist $C>0$ and $h_0>0$ such that, for any $h\in (0,h_0)$,
we have
\[
\hat{\nu}\,
\omega_{\mathrm{min}}(B)^{\frac{2}{k+2}}h^{\frac{2k+2}{k+2}} - C
h^{\frac{6k+8}{3(k+2)}} \leq \lambda_0(H^h) \leq \hat{\nu}\,
\omega_{\mathrm{min}}(B)^{\frac{2}{k+2}}h^{\frac{2k+2}{k+2}} + C
h^{\frac{6k+8}{3(k+2)}}.
\]
\end{theorem}

A similar result was obtained for the bottom of the spectrum of the
Neumann realization of the operator $H^h$ in a bounded domain in
$\RR^2$ by Pan and Kwek~\cite{Pan-Kwek} in the case $k=1$ and by
Aramaki \cite{Aramaki05} in the case $k$ arbitrary odd.

As an immediate consequence of Theorems~\ref{YK:hypersurface}
and~\ref{t:estimate}, we obtain the following statement.

\begin{corollary}
In addition to the assumptions of Theorem~\ref{t:estimate}, suppose
that $M$ is compact. Denote by $\lambda_0(H^h)\leq
\lambda_1(H^h)\leq \lambda_2(H^h)\leq \ldots$ the eigenvalues of the
operator $H^h$ in $L^2(M)$. Then, for integer $m\geq 0$, we have
\[
\lim_{h\to 0} h^{-\frac{2k+2}{k+2}}\lambda_m(H^h)=\hat{\nu}\,
\omega_{\mathrm{min}}(B)^{\frac{2}{k+2}}.
\]
\end{corollary}

In the case when $k=1$ and $|\omega_{0,1}(x)|$ is constant along
$S$, this result was obtained by Montgomery \cite{Mont}. In Theorem
1.4, under some additional assumptions, more generic than in
\cite{Mont}, we obtain stronger estimates for the eigenvalues of the
one-well problem.

Like in the case of the Schr\"odinger operator with electric
potential (see \cite{HelSj5}), one can introduce an internal notion
of magnetic well for a fixed hypersurface $S$ in the zero set of the
magnetic field $\mathbf B$. Such magnetic wells can be naturally
called magnetic miniwells. They are defined by means of the function
$|\omega_{0,1}|$ on $S$. Assuming that there exists a non-degenerate
miniwell on $S$, we prove the existence of gaps in the spectrum of
$H^h$ on intervals of size $h^{\frac{2k+3}{k+2}}$, close to the
bottom $\lambda_0(H^h)$.

\begin{theorem}\label{t:gaps}
Suppose that the operator $H^h$ satisfies condition \eqref{YK:tr1}
with some $\epsilon_0>0$, and that there exists $x_0\in M$ such that
$\mathbf B(x_0)=0$ and in a neighborhood $U$ of $x_0$ the zero set
of $\mathbf B$ is a smooth oriented hypersurface $S$. Suppose
moreover that
there are constants $k\in \ZZ, k>0$ and $C>0$ such that for all
$x\in U$ we have:
\[
C^{-1}d(x,S)^k\leq |B(x)|  \leq C d(x,S)^k\,.
\]
Assume finally that there exists $x_1\in S$ such that $|\omega_{0,1}(x_1)|=
\omega_{\mathrm{min}}(B)$ and, for all $x\in S$ in some neighborhood
of $x_1$
\[
C_1d_S(x,x_1)^2  \leq  |\omega_{0,1}(x)|-\omega_{\mathrm{min}}(B)\leq
C_2d_S(x,x_1)^2.
\]
Then, for any natural $N$, there exist $b_N>0$ and $h_N>0$ such that
the spectrum of $H^h$ in the interval
\[
\left[\hat{\nu}\,
\omega_{\mathrm{min}}(B)^{\frac{2}{k+2}}h^{\frac{2k+2}{k+2}},
\hat{\nu}\,
\omega_{\mathrm{min}}(B)^{\frac{2}{k+2}}h^{\frac{2k+2}{k+2}} +
b_Nh^{\frac{2k+3}{k+2}}\right]
\]
has at least $N$ gaps for any $h\in (0,h_N)$.
\end{theorem}

As an immediate consequence of Theorem~\ref{t:gaps}, we also obtain
upper bounds for the eigenvalues of the one-well problem.

\begin{theorem}\label{t:upperbounds}
Suppose that $M$ is a compact manifold of dimension $n\geq 2$.
Assume that the zero set of $\mathbf B$ is a smooth oriented
hypersurface $S$, and there are constants $k\in \ZZ$, $k>0$, and
$C>0$ such that for all $x\in S$ we have:
\[
C^{-1}d(x,S)^k\leq |B(x)|  \leq C d(x,S)^k\,.
\]
Assume that there exists $x_1\in S$ such that $|\omega_{0,1}(x_1)|=
\omega_{\mathrm{min}}(B)$ and, for all $x\in S$ in some neighborhood
of $x_1$, we have the estimate
\[
C_1d_S(x,x_1)^2  \leq  |\omega_{0,1}(x)|-\omega_{\mathrm{min}}(B)\leq
C_2d_S(x,x_1)^2.
\]

Denote by $\lambda_0(H^h)\leq \lambda_1(H^h)\leq \lambda_2(H^h)\leq
\ldots$ the eigenvalues of the operator $H^h$. Then, for any natural
$m$, there exist $C_m>0$ and $h_m>0$ such that for any $h\in
(0,h_m)$ we have
\[
\lambda_m(H^h) \leq \hat{\nu}\,
\omega_{\mathrm{min}}(B)^{\frac{2}{k+2}}h^{\frac{2k+2}{k+2}} + C_m
h^{\frac{2k+3}{k+2}}.
\]
\end{theorem}

Under the additional assumption that there exists a unique miniwell,
we believe that, using the methods of \cite{Fournais-Helffer06}, one
can
prove the lower bound for the ground state energy $\lambda_0(H^h)$
of the form
\[
\lambda_0(H^h) \geq \hat{\nu}\,
\omega_{\mathrm{min}}(B)^{\frac{2}{k+2}}h^{\frac{2k+2}{k+2}} - C
h^{\frac{2k+3}{k+2}}
\]
and the upper bound for the splitting between $\lambda_0(H^h)$ and
$\lambda_1(H^h)$ of the form
\[
\lambda_1(H^h)-\lambda_0(H^h)\leq C h^{\frac{2k+3}{k+2}}.
\]
For this, we have to know a certain property of the ground state
energy $\lambda_0(\alpha)$ of the operator \eqref{e:Q0}, which we
state as a conjecture.

\begin{conjecture}\label{c:main}
Any minimum of $\lambda_0(\alpha)$ is non-degenerate, that, is, for
any $\alpha_{\rm min}\in \RR$ such that $\lambda_0(\alpha_{\rm
min})=\hat{\nu}$ we have
\begin{equation}\label{e:conjecture}
\frac{\partial^2 \lambda_0}{\partial \alpha ^2}(\alpha_{\rm min})>0.
\end{equation}
\end{conjecture}

Moreover, we believe that one can prove the lower bound for the
splitting between $\lambda_0(H^h)$ and $\lambda_1(H^h)$ of the form
\[
\lambda_1(H^h)-\lambda_0(H^h)\geq C h^{\frac{2k+3}{k+2}},
\]
if, in addition, the following conjecture is true.

\begin{conjecture}\label{c:main2}
There exists a unique $\alpha_{\rm min}\in \RR$ such that
$\lambda_0(\alpha_{\rm min})=\hat{\nu}$.
\end{conjecture}

We refer the reader to Section~\ref{s:model} for discussions.

\paragraph{Acknowledgements}~\\
We would like to thank V. Bonnaillie-No\"el for the help with
numerical computations. The first author wishes to thank A. Morame
for useful discussions on related subjects a few years ago.

\section{Some ordinary differential operators}\label{s:model}
For any $\alpha\in \RR$ and $\beta\in\RR, \beta\neq 0$, consider the
self-adjoint second order differential operator in $L^2(\RR)$ given
by
\[
Q(\alpha, \beta)=-\frac{d^2 }{d t^2}+ \left(\frac{1}{k+1}\beta
t^{k+1}- \alpha \right)^2.
\]
In the context of magnetic bottles, this family of operators (for
$k=1$) first appears in \cite{Mont} (see also \cite{HM}). Denote by
$\lambda_0(\alpha,\beta)$ the bottom of the spectrum of the operator
$Q(\alpha,\beta)$.

In this section, we recall some properties of
$\lambda_0(\alpha,\beta)$, which were established in
\cite{Mont,HM,Pan-Kwek}. First of all, let us remark that
$\lambda_0(\alpha,\beta)$ is a continuous function of $\alpha\in
\RR$ and $\beta\in \RR\setminus\{0\}$. One can see by scaling that,
for $\beta>0$,
\begin{equation}\label{YK:alphabeta}
\lambda_0(\alpha,\beta)= \beta ^{\frac{2}{k+2}} \lambda_0
(\beta^{-\frac{1}{k+2}}\alpha,1)\;.
\end{equation}
A further discussion depends on the parity of $k$.

When $k$ is odd, $\lambda_0(\alpha,1)$ tends to $+\infty$ as $\alpha
\rightarrow -\infty$ by monotonicity. For analyzing its behavior as
$\alpha\rightarrow +\infty$, it is suitable to do a dilation
$t=\alpha^{\frac{1}{k+1}} s$, which leads to the analysis of
\[
\alpha^2 \left(-h^2\frac{d^2}{ds^2} + \left(\frac{s^{k+1}}{k+1}
-1\right)^2 \right)
\]
with $h=\alpha^{-(k+2)/(k+1)}$ small. One can use the semi-classical
analysis (see \cite{CDS} for the one-dimensional case and
\cite{Simon,HSI} for the multidimensional case) to show that
\[
\lambda_0(\alpha,1) \sim (k+1)^{\frac{2k}{k+1}}
\alpha^{\frac{k}{k+1}}\;,\;\mbox{ as } \alpha \rightarrow +\infty\;.
\]
In particular, we see that $\lambda_0(\alpha,1)$ tends to $+\infty$.

When $k$ is even, we have $
\lambda_0(\alpha,1)=\lambda_0(-\alpha,1)$, and, therefore, it is
sufficient to consider the case $\alpha \geq 0$. As $\alpha
\rightarrow +\infty$, semi-classical analysis again shows that
$\lambda_0(\alpha,1)$ tends to $+\infty$.

So in both cases, it is clear that the continuous function
$\lambda_0(\alpha,1)$ is nonnegative
and that there exists (at least one) $\alpha_{\mathrm{min}}\in \RR$ such
that $\lambda_0(\alpha,1)$ is minimal:
\[
\lambda_0(\alpha_{\mathrm{min}},1)=\hat{\nu}.
\]

The results of numerical computations\footnote{performed for us by V.~Bonnaillie-No\"el} for $\alpha_{\mathrm{min}}$,
$\hat\nu $ and the second eigenvalue $\lambda_{1}$ of the operator
$Q(\alpha_{\rm min},1)$ are given (modulo $10^{-2}$) in Table~\ref{t:comp}.

\begin{table}[h]
\begin{center}
\begin{tabular}{|c|c|c|c|c|c|c|c|}
\hline $k$     &    1  &       2  &     3   &    4   &    5 & 6 & 7\\
\hline $\alpha_{\mathrm{min}}$  & 0.35  &    0    &       0.16 &
0
& 0.10 & 0 &   0.07\\
\hline $\hat\nu $  &       0.57  &  0.66 & 0.68 & 0.76 &
0.81 & 0.87 & 0.92\\ \hline $\lambda_{1}$ & 1.98 & 2.50 &
2.61 & 2.98 & 3.18 & 3.47 & 3.66\\ \hline
\end{tabular}
\end{center}
\caption{Numerical results for $\alpha_{\mathrm{min}}$, $\hat\nu $
and  $\lambda_{1}$}\label{t:comp}
\end{table}

In Fig.~\ref{f:even} and~\ref{f:odd}, one can also see the graphs of
the function $\lambda=\lambda_0(\alpha,1)$ and its quadratic
approximation at $\alpha=\alpha_{\mathrm{min}}$:
\[
\lambda^{\mathrm{quad}}(\alpha)=\lambda_0(\alpha_{\rm
min},1)+\frac12 \frac{\partial^2 \lambda_0}{\partial\alpha^2}
(\alpha_{\mathrm{min}},1)(\alpha- \alpha_{\mathrm{min}} )^2\,.
\]

\begin{figure}[h]
\begin{center}
\includegraphics[width=12pc, height=10pc]{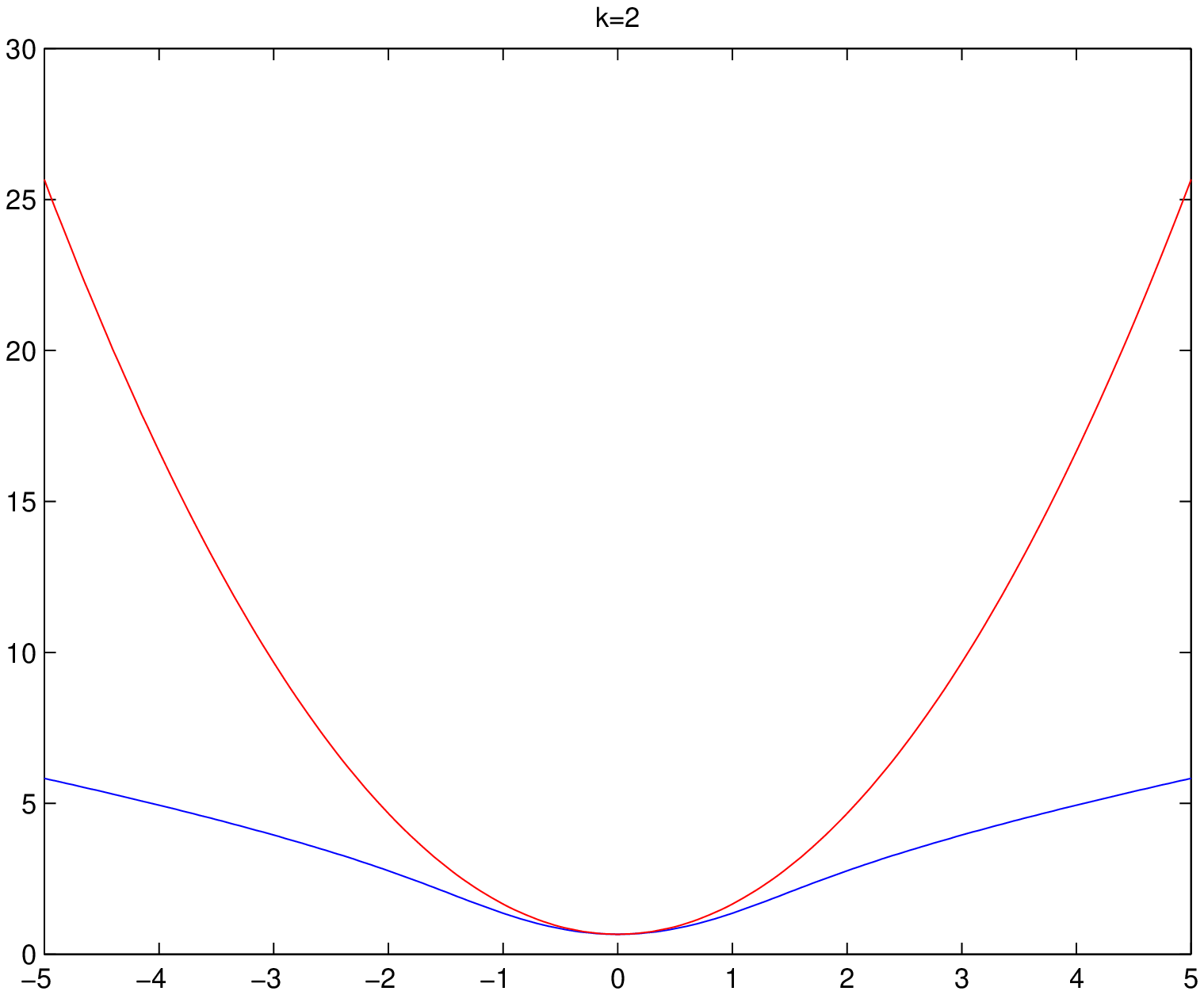}
\includegraphics[width=12pc, height=10pc]{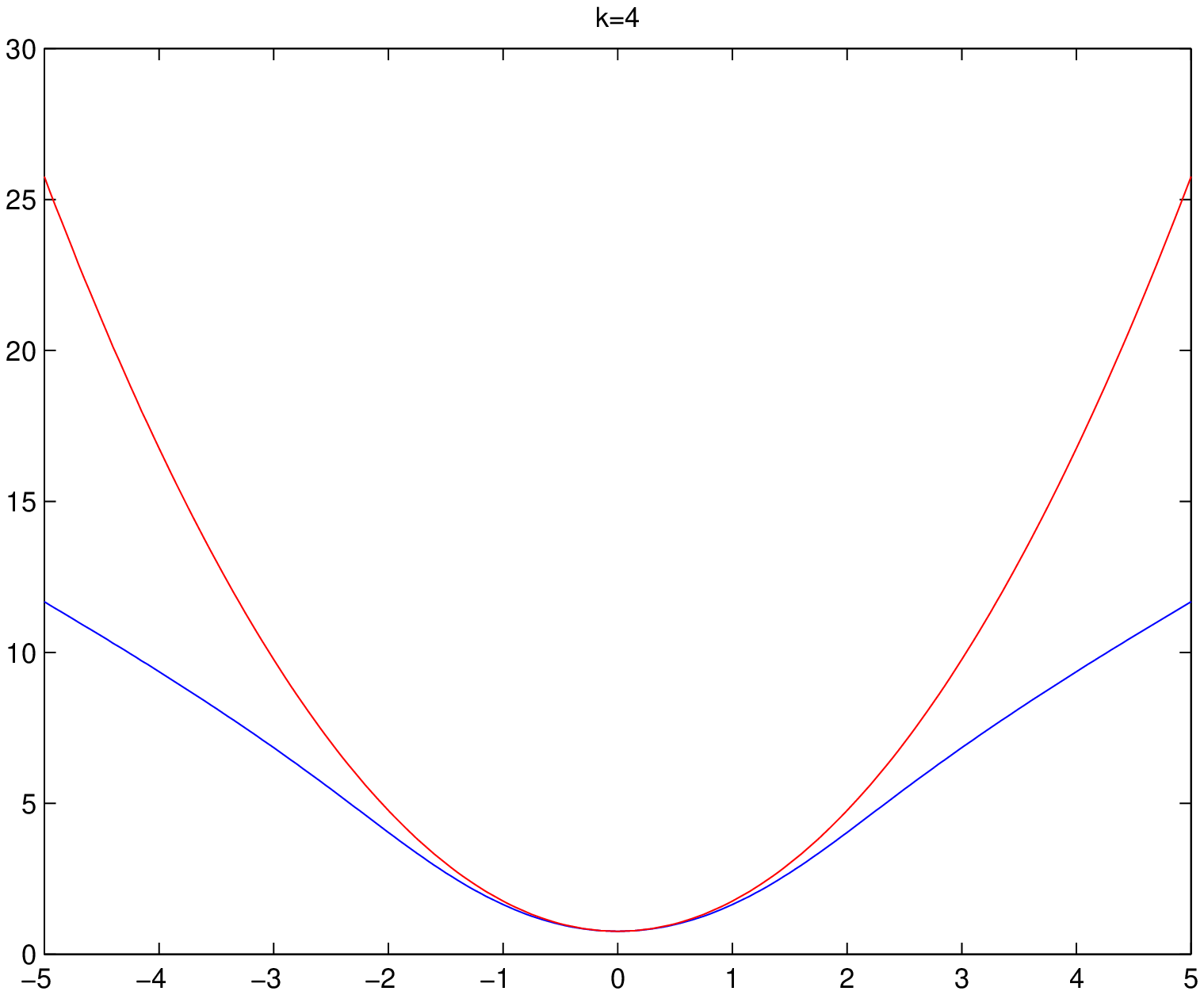}
\end{center}
\caption{$k$ even}\label{f:even}
\end{figure}

\begin{figure}[h]
\begin{center}
\includegraphics[width=12pc, height=10pc]{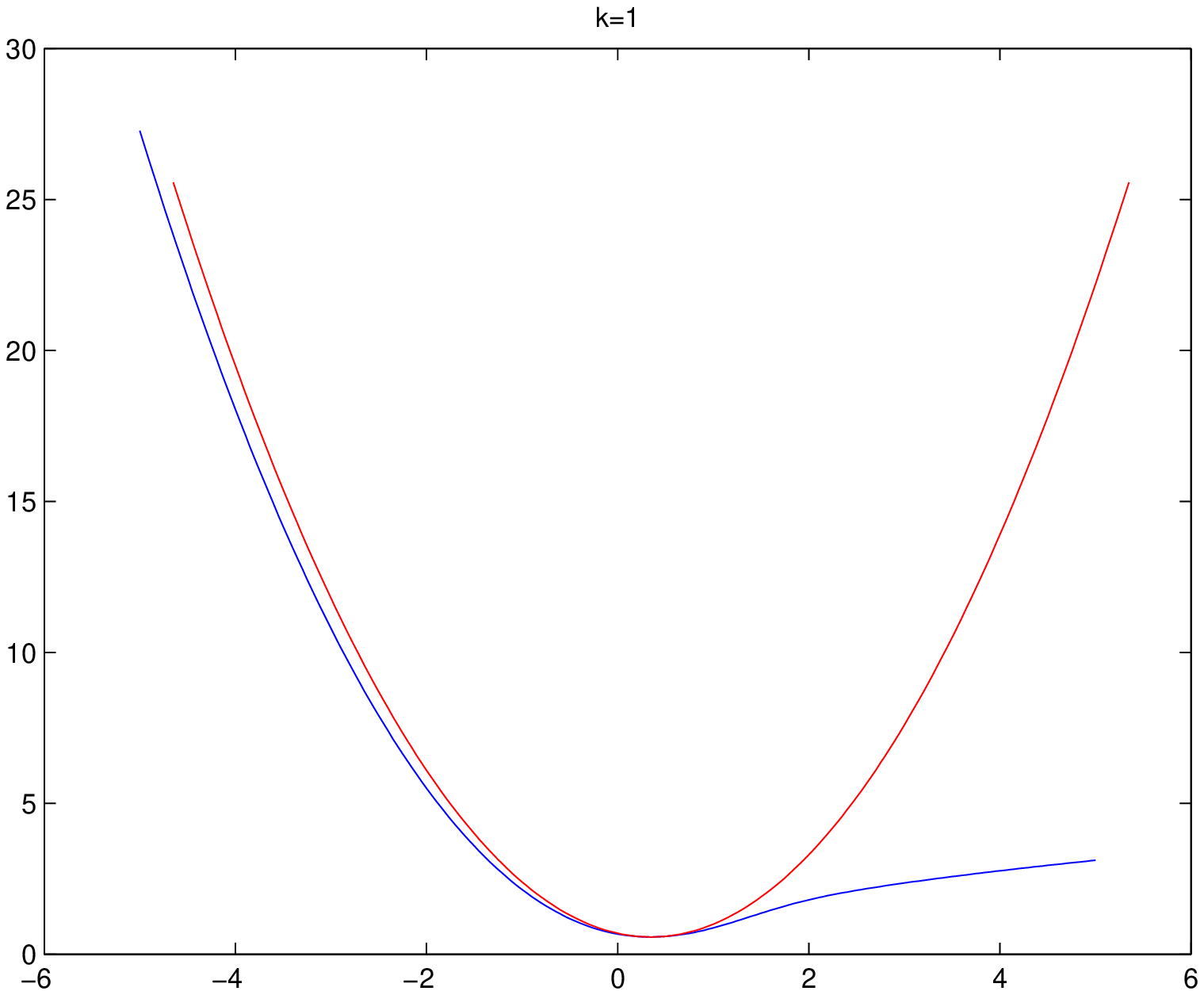}
\includegraphics[width=12pc, height=10pc]{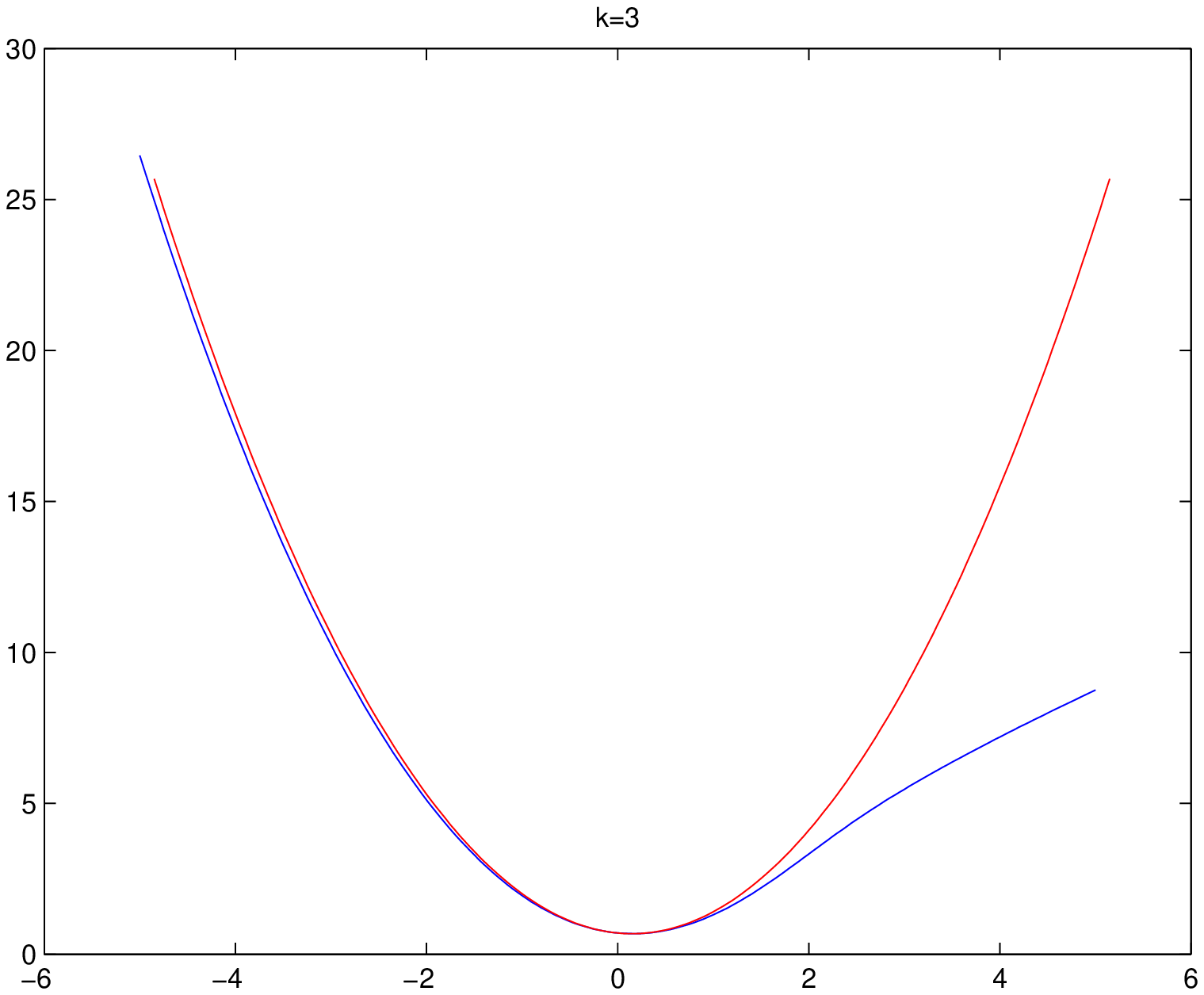}
\end{center}
\caption{$k$ odd} \label{f:odd}
\end{figure}

Numerical computations show that when $k$ is even the minimum is
attained at zero: $\alpha_{\rm min}=0$. They also suggest that the
minimum $\alpha_{\rm min}$ is non degenerate, supporting
Conjecture~\ref{c:main}, and that the second derivative
$\frac{\partial^2 \lambda_0}{\partial\alpha^2}
(\alpha_{\mathrm{min}},1)$ tends as $k$ tends to $\infty$ to $2$.
One more step towards the proof of Conjecture~\ref{c:main} can be
done by the following considerations.

Let $u^0_\alpha\in L^2(\RR)$, $\|u^0_\alpha\|=1$, be the $L^2$
normalized strictly positive eigenvector of the operator
$Q(\alpha,1)$, corresponding to the eigenvalue
$\lambda_0(\alpha,1)$:
\begin{equation}\label{e:Q}
Q(\alpha,1)u^0_\alpha = -\frac{d^2 u^0_\alpha}{d t^2}+
\left(\frac{t^{k+1}}{k+1}- \alpha \right)^2u^0_\alpha =
\lambda_0(\alpha,1) u^0_\alpha.
\end{equation}
One can show that $u^0_\alpha$ depends smoothly on $\alpha$.
Differentiating \eqref{e:Q} with respect to $\alpha$, we obtain
\begin{equation}\label{e:Qlambda0}
Q(\alpha,1)\frac{\partial u^0_\alpha}{\partial\alpha} -2
\left(\frac{t^{k+1}}{k+1}- \alpha \right) u^0_\alpha =
\frac{\partial \lambda_0}{\partial \alpha}(\alpha,1) u^0_\alpha +
\lambda_0(\alpha,1) \frac{\partial u^0_\alpha}{\partial \alpha},
\end{equation}
that implies that
\begin{equation} \label{e:lambda1}
\frac{\partial \lambda_0}{\partial\alpha}(\alpha,1)= -2\int
\left(\frac{t^{k+1}}{k+1}- \alpha \right) (u^0_\alpha(t))^2\,dt
\end{equation}
and
\begin{equation} \label{e:lambda02} \frac{\partial^2
\lambda_0}{\partial\alpha^2}(\alpha,1)=2 -4\int
\left(\frac{t^{k+1}}{k+1}u^0_\alpha(t)- \alpha \right)\frac{\partial
u^0_\alpha}{\partial\alpha} \,dt.
\end{equation}

It follows from \eqref{e:lambda1} that
\begin{equation}
\int \left(\frac{t^{k+1}}{k+1}- \alpha_{\mathrm{min}} \right)
(u^0_{\alpha_{\mathrm{min}}}(t))^2\,dt=0, \label{e:lambda01}
\end{equation}
and, for $k$ odd, $$\alpha_{\mathrm{min}}=\int \frac{t^{k+1}}{k+1}
(u^0_{\alpha_{\mathrm{min}}}(t))^2\,dt>0\,.$$ It has been claimed that
this minimum is unique for $k=1$ in \cite{Pan-Kwek} and for
arbitrary odd $k$ in \cite{Aramaki05}.

It also follows from \eqref{e:Qlambda0} that
\begin{equation}\label{e:Qlambda}
\left(Q(\alpha_{\rm min},1)-\hat\nu\right)\frac{\partial
u^0_\alpha}{\partial\alpha}= 2 \left(\frac{t^{k+1}}{k+1}-
\alpha_{\rm min} \right) u^0_{\alpha_{\mathrm{min}}}.
\end{equation}

Note that the above computations can be made not only for a minimum
point $\alpha_{\mathrm{min}}$, but also for any critical point of
$\lambda_0$.

We will also need the following identity (see \cite{Pan-Kwek},
Proposition 3.5 and the formula (3.14)):
\begin{equation}\label{e:norm}
\left\|\left(\frac{1}{k+1}t^{k+1}- \alpha_{\rm min}\right)
u^0_{\alpha_{\rm min}}\right\|^2=\frac{\hat\nu}{k+2}.
\end{equation}

Based on numerical computations, one can give a proof of
Conjecture~\ref{c:main} for small $k$ as follows. Since
$\frac{\partial u^0_\alpha}{\partial\alpha}$ is orthogonal in
$L^2(\RR,dt)$ to $u^0_\alpha$, by \eqref{e:Qlambda} and
\eqref{e:norm}, we have
\begin{align*}
\left\|\frac{\partial u^0_\alpha}{\partial\alpha}\right\|^2 & \leq
\frac{1}{\lambda_1-\hat\nu} \left(\left(Q(\alpha_{\rm
min},1)-\hat\nu\right)\frac{\partial
u^0_\alpha}{\partial\alpha},\frac{\partial
u^0_\alpha}{\partial\alpha}\right)\\ &  =
\frac{2}{\lambda_1-\hat\nu} \int \left(\frac{t^{k+1}}{k+1}- \alpha_{min}
\right)u^0_{\alpha_{\mathrm{min}}}(t) \frac{\partial
u^0_\alpha}{\partial\alpha}(t)
dt \\
& \leq \frac{2}{\lambda_1-\hat\nu}
\left(\frac{\hat\nu}{k+2}\right)^{1/2} \,\left\|\frac{\partial
u^0_\alpha}{\partial\alpha}\right\|.
\end{align*}
This implies that
\[
\left\|\frac{\partial u^0_\alpha}{\partial\alpha}\right\| \leq
\frac{2}{\lambda_1-\hat\nu} \left(\frac{\hat\nu}{k+2}\right)^{1/2}
\]
and
\[
\int \left(\frac{t^{k+1}}{k+1}- \alpha
\right)u^0_{\alpha_{\mathrm{min}}}(t) \frac{\partial
u^0_\alpha}{\partial\alpha}(t) dt \leq
\frac{2\hat\nu}{(\lambda_1-\hat\nu)(k+2)}\,.
\]
By \eqref{e:lambda02}, it follows that
\[
 \frac{\partial^2 \lambda_0}{\partial\alpha^2}(\alpha_{min},1) \geq 2 -
\frac{8\hat\nu}{(\lambda_1-\hat\nu)(k+2)}= 2 \frac{(k+2)
 \lambda_1 - (k+6) \hat\nu}{ (k+2)(\lambda_1-\hat\nu)} \,.
\]
Hence if for some $k$, we have
\begin{equation}\label{condik}
 (k+2) \lambda_1  > (k+6) \hat\nu\,,
\end{equation}
we deduce that the corresponding minimum $\alpha_{\rm min}$ is non-degenerate.

Using the data given in Table~\ref{t:comp}, one easily verifies that
the condition \eqref{condik} is satisfied for $k=1,\dots,7$. Due to
the accuracy of these numerical computations, one can consider that
 \eqref{e:conjecture}  is safely controlled  for $k=1,\dots,7$.

When $k$ is odd, the potential of the Sturm-Liouville operator
$Q(\alpha, 1)$ is even and, by well-known properties of
Sturm-Liouville operators, any eigenfunction associated with the
$m$-th eigenvalue $\lambda_m(\alpha)$ of $Q(\alpha, 1)$ is even if
$m$ is even and odd if $m$ is odd. In particular, $u^0_\alpha$ is
even for any $\alpha$, and, therefore, $\frac{\partial
u^0_\alpha}{\partial\alpha}$ is even as well. This shows that one
can replace $\lambda_1$ by $\lambda_2$ in the above arguments,
stating that if, for some odd $k$, we have
\[
 (k+2) \lambda_2  > (k+6) \hat\nu\,,
\]
the corresponding minimum $\alpha_{\rm min}$ is non-degenerate.

\section{Estimates for the bottom}
In this section, we will prove Theorem~\ref{t:estimate}. So we
assume that the zero set of the magnetic field $\mathbf B$ is a
smooth oriented hypersurface $S$, and the magnetic field satisfies
the estimate~(\ref{YK:B1}) in a neighborhood $U$ of $S$ with some
$k\in \ZZ$, $k>0$.

Let $G$ be the Riemannian metric on $S$ induced by $g$. Denote by
$dx_{G}$ the corresponding Riemannian volume form on $S$. Let
\[
\omega_{0.0}=i^*_S{\mathbf A}
\]
be the closed one form on $S$ induced by ${\mathbf A}$, where $i_S$
is the embedding of $S$ to $M$. For any $t\in \RR$, let
$P^h_S\left(\omega_{0,0} +\frac{1}{k+1}t^{k+1}\omega_{0,1}\right)$
be a formally self-adjoint operator in $L^2(S, dx_{G})$ defined by
\begin{multline*}
P^h_S\left(\omega_{0,0}+\frac{1}{k+1}t^{k+1}\omega_{0,1}\right)=
\left(ihd+\omega_{0,0}+\frac{1}{k+1}t^{k+1}\omega_{0,1}\right)^*\\
\times\left(ihd+\omega_{0,0}+\frac{1}{k+1}t^{k+1}\omega_{0,1}\right).
\end{multline*}

Consider the self-adjoint operator $H^{h,0}$ in $L^2(\RR\times S,
dt\,dx_{G})$ defined by the formula
\[
H^{h,0}=-h^2\frac{\partial^2 }{\partial t^2}+
P^h_S\left(\omega_{0,0}+\frac{1}{k+1}t^{k+1}\omega_{0,1}\right).
\]
By Theorem~2.7 of \cite{HM}, the operator $H^{h,0}$ has discrete
spectrum.

The quadratic form associated with the operator $H^{h,0}$ is given
by
\begin{multline*}
q^h[u] : = \int_{-\infty}^{+\infty}\int_S
\left[h^2\left|\frac{\partial u}{\partial t}\right|^2+
\left|(ih\,d+\omega_{0,0}+\frac{1}{k+1}t^{k+1}\omega_{0,1})u\right|_{g_0}^2
\right]\,dt\,dx_{g_0}, \\ u\in C^\infty_c(\RR\times S).
\end{multline*}

Without loss of generality, we can assume that $U$ coincides with an
open tubular neighborhood of $S$ and choose a diffeomorphism
\[
\Theta : I\times S\to U,
\]
where $I$ is an open interval $(-\varepsilon_0,\varepsilon_0)$ with
$\varepsilon_0>0$ small enough, such that $\Theta\left|_{\{0\}\times
S}\right.=\mathrm{id}$ and
\begin{equation}\label{e:diffg}
(\Theta^*g-\tilde{g}_0)\left|_{\{0\}\times S}\right.=0,
\end{equation}
where $\tilde{g}_0$ is a Riemannian metric on $I\times S$ given by
\[
\tilde{g}_0=dt^2+G.
\]
By adding to $\mathbf A$ the exact one form $d\phi$, where $\phi$ is
the function satisfying
\begin{gather*}
    N(x)\phi(x)=-\langle N,{\bf A}\rangle(x), \quad x\in U,\\
    \phi(x) =0, \quad x\in S,
\end{gather*}
we may assume that
\begin{equation}\label{e:NA}
\langle N,{\bf A}\rangle(x)=0, \quad x\in U\,.
\end{equation}

Denote by $H^h_D$ the unbounded self-adjoint operator in the Hilbert
space $L^2(D)$ defined by the operator $H^h$ in the domain
$D=\overline{U}$ with the Dirichlet boundary conditions. The
operator $H^h_D$ is generated by the quadratic form
\[
u\mapsto q^h_D [u] : = \int_D |(ih\,d+{\bf A})u|^2\,dx
\]
with the domain
\[
\Dom (q^h_D) = \{ u\in L^2(D) : (ih\,d+{\bf A})u \in L^2\Omega^1(D),
u\left|_{\partial D}\right.=0 \},
\]
where $L^2\Omega^1(D)$ denotes the Hilbert space of $L^2$
differential $1$-forms on $D$ and  $dx$ is the Riemannian volume
form on $D$. Denote by $\lambda_0(H_D^h)$ the bottom of the spectrum
of the operator $H_D^h$. By Theorem 2.1 in \cite{gaps}, there exist
$C, c, h_0>0$ such that for any $h\in (0,h_0]$ we have
\[
|\lambda_0(H^{h})-\lambda_0(H_D^h)|< Ce^{-c/\sqrt{h}}.
\]
It can be seen from the proof of Theorem~2.7 in \cite{HM} that, if
there exist $h_0>0$, a family of functions $w^h\in
C^\infty_c(I\times S)$, $h\in (0,h_0]$, and a function
$\lambda^0(h)$ defined on $h\in (0,h_0]$ such that
\[
\lambda^0(h)\leq \widehat{C} h^{(2k+2)/(k+2)}, \quad h\in (0,h_0],
\]
and
\[
\|(H^{h,0}-\lambda^0(h))w^h\|\leq C_1h^{(2k+3)/(k+2)}\|w^h\|,\quad
h\in (0,h_0]\,,
\]
with some positive constants $\widehat{C}$ and $C_1$, then there
exist $h_1\leq h_0 $ and a positive constant $C_2$ such that
\[
\|(H_D^h-\lambda^0(h))v^h\|\leq C_2h^{(2k+3)/(k+2)}\|v^h\|\,,\,
\forall h\in ]0,h_1]\,,
\]
where $v^h=(\Theta^{-1})^*w^h\in C^\infty_c(U)\,$. So, in order to
complete the proof, it sufficient to prove that there exist $h_0>0$
and positive constants $C_1$ and $C_2$ such that, for all $h\in
]0,h_0]$,
\[
\hat{\nu}\,
\omega_{\mathrm{min}}(B)^{\frac{2}{k+2}}h^{\frac{2k+2}{k+2}} - C_1
h^{\frac{6k+8}{3(k+2)}} \leq \lambda_0(H^{h,0}) \leq \hat{\nu}\,
\omega_{\mathrm{min}}(B)^{\frac{2}{k+2}}h^{\frac{2k+2}{k+2}} + C_2
h^{\frac{6k+8}{3(k+2)}}\,.
\]
The desired upper bound for $\lambda_0(H^{h,0})$ follows from the
construction of approximate eigenfunctions of $H^{h,0}$ given in
\cite{Qmath10}. It remains to prove the lower bound for
$\lambda_0(H^{h,0})$.

We will localize the problem in two scales.\\

{\bf  First}, we choose a
covering of $S$ by local coordinate charts, $S=\bigcup_{m=1}^d U_m$.
Let $\chi_m\in C^\infty(S)$ be a partition of unity subordinate to
this covering so that ${\rm supp}\,\chi_m\subset U_m$ for each $m$
and
\[
\sum_{m=1}^d\chi^2_m(x)=1, \quad x\in S.
\]
There exists a $C^\infty$ real valued function $\varphi_m$ such that
on $U_m$
\[
\omega_{0,0}=d\varphi_m.
\]
Therefore, for $v\in C^\infty_c(U_m)$, we obtain
\begin{multline*}
q^h[v] = \int_{-\infty}^{+\infty}\int_S \Big[h^2\left|\frac{\partial
v}{\partial t}\right|^2\\ +
\left|(ih\,d+\frac{1}{k+1}t^{k+1}\omega_{0,1})\,\exp
\left(-\frac{i}{h}\varphi_m\right)v\right|_{G}^2 \Big]\,dt\,dx_{G}.
\end{multline*}

{\bf Secondly}, for a fixed $m=1,2,\ldots,d$, by scaling a standard partition
of unity in $\RR^{n-1}$, we can find a partition of unity satisfying
\[
\sum_{j}|\chi^h_{m,j}(x)|^2=1, \quad x\in U_m.
\]
\[
\sum_{j}|\nabla\chi^h_{m,j}(x)|^2\leq
C\varepsilon_0^{-2}h^{-2\beta}, \quad x\in U_m.
\]
and
\[
{\rm supp}\,\chi^h_{m,j}\subset Q_{m,j}:=\{x\in\RR^{n-1} :
|x-y_{m,j}|\leq \varepsilon_0 h^{\beta}\},
\]
where $\beta>0$ is a parameter which will be determined later.

Choose an arbitrary $z_{m,j}\in Q_{m,j}$. Let $g_{m,j}=G(z_{m,j})$
be a flat Euclidean metric in $\RR^{n-1}$. For any $\ell$ and $p$,
we have
\begin{equation}\label{e:g}
|G^{\ell p}(x)-g_{m,j}^{\ell p}(x)|=O(h^{\beta})\ \text{as}\ h\to 0,
\quad x\in Q_{m,j}.
\end{equation}
We also have a similar estimate for $\omega_{0,1}$:
\begin{equation}\label{e:omega}
|\omega_{0,1}(x)-\omega_{0,1}(z_{m,j})|=O(h^{\beta})\ \text{as}\
h\to 0, \quad x\in Q_{m,j}.
\end{equation}

Consider the self-adjoint operator $H^{h,0}_{m,j}$ in $L^2(\RR\times
Q_{m,j}, dt\,dx_{g_{m,j}})$ defined by the formula
\begin{multline*}
H^{h,0}_{m,j}=-h^2\frac{\partial^2 }{\partial t^2}\\ +
\left(ihd+\frac{1}{k+1}t^{k+1}\omega_{0,1}(z_{m,j})\right)^*
\left(ihd+\frac{1}{k+1}t^{k+1}\omega_{0,1}(z_{m,j})\right).
\end{multline*}
The quadratic form associated with $H^{h,0}_{m,j}$ is given by
\begin{multline*}
q^h_{m,j}[w] : = \int_{-\infty}^{+\infty}\int_{Q_{m,j}}
\Big[h^2\left|\frac{\partial w}{\partial t}\right|^2 \\ +
\left|(ih\,d+\frac{1}{k+1}t^{k+1}\omega_{0,1}(z_{m,j}))w\right|_{g_{m,j}}^2
\Big] |G(z_{m,j})|^{1/2}\,dt\,dx, \\ w\in C^\infty_c(\RR\times
Q_{m,j}).
\end{multline*}

\begin{lemma}
For any $w\in C^\infty_c(\RR\times Q_{m,j})$, we have
\[
q^h_{m,j}[w]\geq
\hat{\nu}\,|\omega_{0,1}(z_{m,j})|^{\frac{2}{k+2}}h^{\frac{2k+2}{k+2}}
\|w\|^2.
\]
\end{lemma}

\begin{proof}
Using the simple scaling $t=h^{1/(k+2)}\tau$, we can write
\begin{multline*}
q^h_{m,j}[w] = h^{\frac{2k+2}{k+2}}
\int_{-\infty}^{\infty}\int_{Q_{m,j}} \Big[\left|\frac{\partial
w}{\partial \tau}\right|^2 \\ +
\left|(ih^{\frac{1}{k+2}}\,d+\frac{1}{k+1}\tau^{k+1}\omega_{0,1}(z_{m,j}))w\right|_{g_{m,j}}^2
\Big]|G(z_{m,j})|^{1/2}\,d\tau\,dx.
\end{multline*}
There exists a $g_{m,j}$-orthonormal frame $\{X_1,X_2, \ldots,
X_{n-1}\}$ in $\RR^{n-1}$ such that
\[
\omega_{0,1}(z_{m,j})(X_1)=|\omega_{0,1}(z_{m,j})|, \quad
\omega_{0,1}(z_{m,j})(X_\ell)=0,\quad  l\geq 2.
\]
Then, by a linear change of variables in $\RR^{n-1}$, we obtain
\begin{align*}
q^h_{m,j}[w] = & h^{\frac{2k+2}{k+2}}
\int_{-\infty}^{\infty}\int_{Q_{m,j}} \Big[\left|\frac{\partial
w}{\partial \tau}\right|^2 +
\left|ih^{\frac{1}{k+2}}\,\frac{\partial w}{\partial y_1} +
\frac{1}{k+1}\tau^{k+1} |\omega_{0,1}(z_{m,j})|w\right|^2 \\ &
+h^{\frac{2}{k+2}} \sum_{l=2}^{n-1} \left| \frac{\partial
w}{\partial y_\ell}\right|^2 \Big]\,d\tau\,dy_1\,dy_2\ldots
dy_{n-1}\\ \geq & h^{\frac{2k+2}{k+2}}
\int_{-\infty}^{\infty}\int_{\RR^{n-1}} \Big[\left|\frac{\partial
w}{\partial \tau}\right|^2 \\ & +
\left|ih^{\frac{1}{k+2}}\,\frac{\partial w}{\partial y_1} +
\frac{1}{k+1}\tau^{k+1}
|\omega_{0,1}(z_{m,j})|w\right|^2\Big]\,d\tau\,dy_1\,dy_2\ldots
dy_{n-1}.
\end{align*}
Denote by $\widehat{w}(\eta_1,y_2,\ldots,y_{n-1})$ the partial
Fourier transform of $w$ in the $y_1$-variable. We have
\begin{multline*}
\int_{\RR^{n}} \Big[\left|\frac{\partial w}{\partial \tau}\right|^2
+ \left|ih^{\frac{1}{k+2}}\,\frac{\partial w}{\partial y_1} +
\frac{1}{k+1}\tau^{k+1}
|\omega_{0,1}(z_{m,j})|w\right|^2\Big]\,d\tau\,dy_1\,dy_2\ldots
dy_{n-1}\\
= \int_{\RR^{n}} \Big[\left|\frac{\partial \widehat{w}}{\partial
\tau}\right|^2 + \left( \frac{1}{k+1}\tau^{k+1}
|\omega_{0,1}(z_{m,j})|
-h^{\frac{1}{k+2}}\eta_1\right)^2|\widehat{w}|^2\Big]\,d\tau\,d\eta_1\,dy_2\ldots
dy_{n-1}.
\end{multline*}
For any fixed $\eta_1, y_2,\ldots, y_{n-1}$, the expression
\[
\int_{-\infty}^{\infty} \Big[\left|\frac{\partial
\widehat{w}}{\partial \tau}\right|^2 + \left(
\frac{1}{k+1}\tau^{k+1} |\omega_{0,1}(z_{m,j})|
-h^{\frac{1}{k+2}}\eta_1\right)^2|\widehat{w}|^2\Big]\,d\tau
\]
is the quadratic form of the operator $Q(\alpha,\beta)$ with
\[
\alpha=h^{\frac{1}{k+2}}\eta_1, \quad \beta=|\omega_{0,1}(z_{m,j})|
\]
evaluated on $\widehat{w}$. Therefore, we have
\begin{multline*}
\int_{-\infty}^{\infty} \Big[\left|\frac{\partial
\widehat{w}}{\partial \tau}\right|^2 + \left(
\frac{1}{k+1}\tau^{k+1} |\omega_{0,1}(z_{m,j})|
-h^{\frac{1}{k+2}}\eta_1\right)^2|\widehat{w}|^2\Big]\,d\tau\\ \geq
\lambda_0 (h^{\frac{1}{k+2}}\eta_1,|\omega_{0,1}(z_{m,j})|)
\int_{-\infty}^{\infty}|\widehat{w}|^2\,d\tau \geq
\hat{\nu}\,|\omega_{0,1}(z_{m,j})|^{\frac{2}{k+2}}
\int_{-\infty}^{\infty}|\widehat{w}|^2\,d\tau,
\end{multline*}
that immediately completes the proof.
\end{proof}

On the other hand, using the inequality
\[
2|ab|\leq \varepsilon^2a^2+\varepsilon^{-2}b^2\,,
\]
with $\varepsilon=h^\rho$ (here $\rho>0$ is a parameter which will
be determined later) and estimates \eqref{e:g} and \eqref{e:omega},
 we get, for any $w\in C^\infty_c(\RR\times Q_{m,j})$,
\begin{align*}
q^h[w] = & (1+O(h^\beta)) \int_{-\infty}^{+\infty}\int_{Q_{m,j}}
\Big[h^2\left|\frac{\partial w}{\partial t}\right|^2\\ & +
\left|(ih\,d+\frac{1}{k+1}t^{k+1}\omega_{0,1})(\exp
\left(-\frac{i}{h}\varphi_m\right)w\right|_{g_{m,j}}^2
\Big]|G(z_{m,j})|^{1/2}\,dt\,dx
\\ \geq & (1-h^{2\rho})(1+O(h^\beta)) q^h_{m,j}[\exp
\left(-\frac{i}{h}\varphi_m\right) w]\\
& - h^{-2\rho} (1+O(h^\beta)) \int_{-\infty}^{+\infty}\int_{Q_{m,j}}
\frac{1}{(k+1)^2}|t|^{2k+2}\\ & \times \left|
(\omega_{0,1}(x)-\omega_{0,1}(z_{m,j}))w(t,x)\right|_{g_{m,j}}^2
|G(z_{m,j})|^{1/2} \,dt\,dx\\  \geq & (1-h^{2\rho})(1+O(h^\beta))
q^h_{m,j}[\exp
\left(-\frac{i}{h}\varphi_m\right) w]\\
& - \frac{1}{(k+1)^2} h^{2\beta-2\rho}  (1+O(h^\beta)) \\ & \times
\int_{-\infty}^{\infty}\int_{Q_{m,j}}  |t|^{2k+2}|w(t,x)|^2
|G(x)|^{1/2} \,dt\,dx.
\end{align*}
Therefore, using the IMS formula, for any $v\in C^\infty_c(\RR\times
U_m)$, we obtain
\begin{align*}
q^h[v] = & \sum_{j} q^h[\chi_{m,j}v]-h^2\sum_{j} \|
|\nabla\chi_{m,j}|v\|^2 \\
\geq & \left[ \left(1-h^{2\rho} \right) (1+O(h^\beta))
\hat{\nu}\,\omega_{\mathrm{min}}(B)^{\frac{2}{k+2}}h^{\frac{2k+2}{k+2}}
- C h^{2-2\beta}\right]\|v\|^2\\
 & -  \frac{1}{(k+1)^2} h^{2\beta-2\rho} (1+O(h^\beta)) \left\|
|t|^{k+1} v\right\|^2,
\end{align*}
and, furthermore, for any $u\in C^\infty_c(\RR\times S)$, we get
\begin{align*}
q^h[u] = &\sum_{m=1}^dq^h[\chi_mu]-h^2\sum_{m=1}^d\|
|\nabla\chi_m|u\|^2 \\
\geq & \sum_{m=1}^dq^h[\chi_mu]- C h^2\sum_{m=1}^d\|u\|^2\\
\geq & \left[ \left(1-h^{2\rho} \right) (1+O(h^\beta))
\hat{\nu}\,\omega_{\mathrm{min}}(B)^{\frac{2}{k+2}}h^{\frac{2k+2}{k+2}}
- C h^{2-2\beta}\right]\|u\|^2\\
& - \frac{1}{(k+1)^2} h^{2\beta-2\rho}  (1+O(h^\beta)) \left\|
|t|^{k+1} u\right\|^2 .
\end{align*}

It is shown in \cite{HM} (see Remark~5.2) that, given  some constant
 $\widehat C$, then there exists, for any $k\in \mathbb N$,  a constant $C_k$ such that, for $h\in (0,1]$  and any
 eigenvalue $\lambda^0(h)$ of $H^{h,0}$  such that
\[
\lambda^0(h)\leq {\widehat C} h^{(2k+2)/(k+2)}\,,
\]
then any corresponding eigenfunction $U^h$ satisfies
\[
\| |t|^{k+1} u^h \| \leq C_k h^{(k+1)/(k+2)}\|u^h\|\,.
\]
Therefore, we obtain, for some new constants $C_0,\,C_1,\,C_2,\,C$,
\[
\lambda^0(h) \geq
\left(\hat{\nu}\,\omega_{\mathrm{min}}(B)^{\frac{2}{k+2}}-C_0h^{\beta}
-C_1h^{2\rho} - C_2 h^{2\beta-2\rho}\right) h^{\frac{2k+2}{k+2}} - C
h^{2-2\beta}\,.
\]
Putting $\beta=\frac{2}{3(k+2)}, \rho=\frac{1}{3(k+2)}$, this achieves
the proof of the lower bound.

\section{Existence of gaps}
In this section we will prove Theorem~\ref{t:gaps}.

\subsection{A model operator}
First, we will only assume that there exists $x_0\in M$ such that
$\mathbf B(x_0)=0$ and in a neighborhood $U$ of $x_0$ the zero set
of $\mathbf B$ is a smooth oriented hypersurface $S$, and the
magnetic field satisfies in $U$ the estimate~(\ref{YK:B1}) with some
$k\in \ZZ$, $k>0$. As above, we will assume that $U$ coincides with
an open tubular neighborhood of $S$ and choose a diffeomorphism
$\Theta : I\times S\to U$, where $I$ is an open interval
$(-\varepsilon_0,\varepsilon_0)$ with $\varepsilon_0>0$ small
enough, such that $\Theta\left|_{\{0\}\times S}\right.=\mathrm{id}$.
We will also assume that the conditions \eqref{e:diffg} and
\eqref{e:NA} hold.

Suppose that $S$ admits a coordinate system with coordinates
$(s_1,\ldots,s_{n-1})$. Then we have a coordinate system in $U$ with
coordinates $X=(X_0,X_1,\ldots, X_{n-1})$ with $X_0=t\in I$ and
$X_j=s_j, j=1,2,\ldots,n-1$. Thus, $S$ is given by the equation
$t=0$. We have
\[
H^h=\frac{1}{\sqrt{|g(X)|}}\sum_{0\leq \alpha,\beta\leq
n-1}\nabla^h_\alpha \left(\sqrt{|g(X)|} g^{\alpha\beta}(X)
\nabla^h_\beta\right),
\]
where
\[
\nabla^h_\alpha= i h \frac{\partial}{\partial X_\alpha}+A_\alpha(X),
\quad \alpha =0,1,\ldots,n-1.
\]
or, equivalently,
\begin{equation}\label{e:Hh}
H^h=\sum_{0\leq \alpha,\beta\leq n-1} g^{\alpha\beta}(X)
\nabla^h_\alpha\nabla^h_\beta+i h \sum_{0\leq \alpha\leq
n-1}\Gamma^\alpha \nabla^h_\alpha,
\end{equation}
where
\[
\Gamma^\alpha= \frac{1}{\sqrt{|g(X)|}}\sum_{0\leq \beta\leq
n-1}\frac{\partial}{\partial X_\beta} \left(\sqrt{|g(X)|}
g^{\beta\alpha}(X)\right), \quad \alpha =0,1,\ldots,n-1.
\]
By \eqref{e:diffg} and \eqref{e:NA}, we have
\[
A_0(t,s)=0,
\]
and
\[
g=dt^2+G+O(t), \quad t\to 0,
\]
where $G$ is the induced metric on $S$. So we can write
\begin{align*}
g_{00}(t,s)& = 1+\dot{g}_{00}(s)t+O(t^2), \\
g_{0j}(t,s)& =\dot{g}_{0j}(s)t + O(t^2), \\
g_{j\ell}(t,s)& =G_{j\ell}(s)+\dot{g}_{j\ell}(s)t+O(t^2).
\end{align*}
In particular, we have
\[
|g(t,s)|=|G(s)|+O(t), \quad t\to 0\,.
\]
 By assumption, we can write
\[
A_j(t,s)=\omega_{0,0}^{(j)}(s)+\frac{1}{k+1}\omega_{0,1}^{(j)}(s)t^{k+1}
+\frac{1}{k+2}\omega_{0,2}^{(j)}(s)t^{k+2}+O(t^{k+3}), \quad t\to 0\,.
\]

Suppose that a family $u^h\in C^\infty_c(U), h\in (0,h_0],$
satisfies the following assumptions.\\
 For any real $m>0$, there exists $C_m>0$ and
$h_m>0$ such that, for any $h\in (0,h_m)$, we have
\begin{align}
\|t^m u^h\| & \leq C_m h^{\frac{m}{k+2}}\|u^h\|, \label{e:uh} \\
\|t^m \nabla^h_j u^h\|+h\|t^m \frac{\partial u^h}{\partial t}\| &
\leq C_m h^{\frac{m+k+1}{k+2}}\|u^h\|\,, \label{e:u1h}
\end{align}
and
\begin{multline}
\|t^m \nabla^h_j\nabla^h_\ell u^h\|+h\left(\| t^m \frac{\partial
}{\partial t} \nabla^h_j\| + \|t^m \nabla^h_j
\frac{\partial}{\partial t} u^h\|\right)+ h^2\|t^m
\frac{{\partial}^2 u^h}{\partial t^2}\|\\ \leq C_m
h^{\frac{m+2k+2}{k+2}}\|u^h\|\,. \label{e:u2h}
\end{multline}

As shown in \cite{HM}, a family $u^h\in C^\infty_c(U)$, $h\in
(0,h_0)$, such that, for some $C>0$, we have
\[
(H^hu^h,u^h)\leq Ch^{\frac{2k+2}{k+2}}\|u^h\|^2, \quad h\in (0,h_0)\,,
\]
satisfies the conditions~\eqref{e:uh}, \eqref{e:u1h} and
\eqref{e:u2h}.

By \eqref{e:Hh}, it follows that
\begin{multline*}
H^hu^h= - h^2 g^{00}(t,s) \frac{{\partial}^2 u^h}{\partial t^2}+ i h
\sum_{1\leq j\leq n-1} g^{0j}(t,s) \left[ \frac{\partial }{\partial
t} \nabla^h_j + \nabla^h_j \frac{\partial}{\partial t}\right] u^h \\
+ \sum_{1\leq j,\ell\leq n-1} g^{j\ell}(t,s) \nabla^h_j\nabla^h_\ell
u^h - h^2\Gamma^0 \frac{\partial u^h}{\partial t} - h \sum_{1\leq
j\leq n-1}\Gamma^j \nabla^h_ju^h\,.
\end{multline*}
Observe that the first and the third terms can only contribute to
the terms of order $h^{\frac{2k+2}{k+2}}$. More precisely, we have
\begin{multline*}
H^hu^h= - h^2 \frac{{\partial}^2 u^h}{\partial t^2} +\sum_{1\leq
j,\ell\leq n-1} G^{j\ell}(s) \left(i h \frac{\partial}{\partial
s_j}+
\omega_{0,0}^{(j)}(s)+\frac{1}{k+1}\omega_{0,1}^{(j)}(s)t^{k+1}\right)
  \\ \times \left(i h \frac{\partial}{\partial s_\ell}+
\omega_{0,0}^{(\ell)}(s)+\frac{1}{k+1}\omega_{0,1}^{(\ell)}(s)t^{k+1}\right)
u^h+ O(h^{\frac{2k+3}{k+2}})\,.
\end{multline*}

This fact was stated in \cite[Theorem 2.7]{HM}. For the proof of
Theorem~\ref{t:gaps}, we have to improve the remainder
$O(h^{\frac{2k+3}{k+2}})$ in the last formula. For this purpose we
will take into account further terms in the expansion of $H^hu^h$ in
powers of $h^{\frac{1}{k+2}}$. This leads us to introduce a new
model operator $H^h_0$ given by
\begin{align*}
H^h_0 = & -h^2
 \frac{{\partial}^2}{\partial t^2} - h^2 \dot{g}^{00}(s) t \frac{{\partial}^2}{\partial
 t^2}\\
 & + 2 i h \sum_{1\leq j\leq n-1} \dot{g}^{0j}(s) t \left(i h
\frac{\partial}{\partial s_j}+
\omega_{0,0}^{(j)}(s)+\frac{1}{k+1}\omega_{0,1}^{(j)}(s)t^{k+1}\right)
\frac{\partial}{\partial t}\\
& + i h \sum_{1\leq j\leq n-1} \dot{g}^{0j}(s)
\omega_{0,1}^{(j)}(s)t^{k+1}\\
&+ \sum_{1\leq j,\ell\leq n-1} G^{j\ell}(s)\Big(i h
\frac{\partial}{\partial s_j}+
\omega_{0,0}^{(j)}(s)+\frac{1}{k+1}\omega_{0,1}^{(j)}(s)t^{k+1}
\\ & +\frac{1}{k+2}\omega_{0,2}^{(j)}(s)t^{k+2}\Big)
\Big(i h \frac{\partial}{\partial s_\ell}+
\omega_{0,0}^{(\ell)}(s)+\frac{1}{k+1}\omega_{0,1}^{(\ell)}(s)t^{k+1}\\
& +\frac{1}{k+2} \omega_{0,2}^{(j)}(s)t^{k+2}\Big)\\ & + \sum_{1\leq
j,\ell\leq n-1} \dot{g}^{j\ell}(s)t \left(i h
\frac{\partial}{\partial s_j}+
\omega_{0,0}^{(j)}(s)+\frac{1}{k+1}\omega_{0,1}^{(j)}(s)t^{k+1}\right)
\\ & \times \left(i h \frac{\partial}{\partial s_\ell}+
\omega_{0,0}^{(\ell)}(s)+\frac{1}{k+1}\omega_{0,1}^{(\ell)}(s)t^{k+1}\right) \\
 & - h^2\Gamma^0_0(s) \frac{\partial}{\partial t}- h \sum_{1\leq j\leq n-1}\Gamma^j_0(s) \left(i h
\frac{\partial}{\partial s_j} +
\omega_{0,0}^{(j)}(s)+\frac{1}{k+1}\omega_{0,1}^{(j)}(s)t^{k+1}\right)\,.
\end{align*}

\begin{lemma}\label{l:model}
Suppose that a family $u^h\in C^\infty_c(U)$, $h\in (0,h_0)$,
satisfies the conditions~\eqref{e:uh}, \eqref{e:u1h} and
\eqref{e:u2h}. Then, there exists $C>0$ such that, for any $h\in
(0,h_0)$, we have
\[
\|H^hu^h-H^h_0u^h\| \leq Ch^2\|u^h\|\,.
\]
\end{lemma}

\subsection{Approximate eigenfunctions: main result}
Now we additionally assume that there exists $x_1\in S$ such that
$|\omega_{0,1}(x_1)|= \omega_{\mathrm{min}}(B)$, a neighborhood
$\mathcal V$ of $x_1$ in $S$ and a constant $C_1>0$ such that, for
all $x\in \mathcal V$,
\[
C_1d_S(x,x_1)^2  \leq |\omega_{0,1}(x)|-\omega_{\mathrm{min}}(B)\leq
C_1d_S(x,x_1)^2\,.
\]
Take normal coordinate system $f: U(x_1) \subset S \to \RR^{n-1}$ on
$S$ defined in a neighborhood $U(x_1)$ of $x_1$, where $f(U(x_1))
=B(0,r)$ is a ball in $\RR^{n-1}$ centered at the origin and
$f(x_1)=0$. As above, we will denote local coordinates by
$s=(s_1,s_2,\ldots,s_{n-1})$. Note that
\[
\omega_{\mathrm{min}}(B)=\left(\sum\limits_{j=1}^{n-1}|\omega_{0,1}^{(j)}(0)|^2\right)^{1/2}.
\]

Consider the Euclidean space $\RR^{n-1}$ with coordinates
$(\sigma_1,\sigma_2,\ldots,\sigma_{n-1})$.  Take the unit vector
\[
e_\omega=\frac{1}{\omega_{\mathrm{min}}(B)}\left(\omega_{0,1}^{(1)}(0),
\ldots,\omega_{0,1}^{(n-1)}(0)\right)\in \RR^{n-1},
\]
and complete it to an orthonormal base $(e_1=e_\omega, e_2, \ldots,
e_{n-1})$ in $\RR^{n-1}$. Denote by ${\widehat e}_\omega, {\widehat
e}_2, \ldots, {\widehat e}_{n-1}$ the corresponding first order
differential operators with constant coefficients in $\RR^{n-1}$. In
particular, we have
\[
{\widehat e}_\omega = \frac{1}{\omega_{\mathrm{min}}(B)}
\sum_{j=1}^{n-1} \omega_{0,1}^{(j)}(0) \frac{\partial}{\partial
\sigma_j}\,.
\]
The Laplacian $\Delta=-\sum \frac{\partial^2}{\partial \sigma_j^2}$
in $\RR^{n-1}$ can be written as
\[
\Delta=\Delta_{\omega}+\Delta_{\omega^\bot}\,,
\]
where
\[
\Delta_{\omega}=-{\widehat e}^2_\omega, \quad
\Delta_{\omega^\bot}=-\sum_{j=2}^{n-1} {\widehat e}^2_j\,.
\]
Consider a second order differential operator $K$ in $\RR^{n-1}$
given by
\begin{equation}\label{e:defK}
K=\frac 12 \frac{\partial^2 \lambda_0}{\partial\alpha^2}
(\alpha_{\rm min},1) \Delta_\omega + \Delta_{\omega^\bot} +
 \sum_{r,m} \Omega_{rm} \sigma_r \sigma_m +A\,,
\end{equation}
where
\[
\Omega_{rm}=\omega_{\mathrm{min}}(B)^{-\frac{2k+2}{k+2}}\Big[
\frac{\hat\nu}{2(k+2)}\frac{\partial^2 |\omega_{0,1}|^2}{\partial
s_r\partial s_m}(0)+ \alpha_{\rm min}^2 \sum_j \frac{\partial
\omega_{0,1}^{(j)}}{\partial s_r}(0) \frac{\partial
\omega_{0,1}^{(j)}}{\partial s_m}(0) \Big]
\]
and
\[
\begin{split}
A =  & -\dot{g}^{00}(0) \int \tau \frac{{\partial}^2
u^0_\alpha}{\partial
 \tau^2}(\tau) u^0_{\alpha_{\mathrm{min}}}(\tau) d\tau
+ i \omega_{\mathrm{min}}(B)^{-1} \sum_{j=1}^{n-1} \frac{\partial
\omega_{0,1}^{(j)}}{\partial
\sigma_j}(0) \alpha_{\rm min} \\
& +2\omega_{\mathrm{min}}(B)^{-2} \sum_{j=1}^{n-1}
\omega_{0,1}^{(j)}(0)\omega_{0,2}^{(j)}(0)\int
\frac{\tau^{k+2}}{k+2} \Big(\frac{\tau^{k+1}}{k+1}-\alpha_{\rm min}
\Big)(u^0_{\alpha_{\mathrm{min}}}(\tau))^2 d\tau\\
& + \omega_{\mathrm{min}}(B)^{-2} \sum_{1\leq j,\ell\leq n-1}
\dot{g}^{j\ell}(0) \omega_{0,1}^{(j)}(0) \omega_{0,1}^{(\ell)}(0)
\\ & \times \int \tau \left(\frac{\tau^{k+1}}{k+1}-\alpha_{\rm min}
\right)^2(u^0_{\alpha_{\mathrm{min}}}(\tau))^2 d\tau\,.
\end{split}
\]

It should be noted that we have indeed the operator $K=K_{\alpha_{\mathrm{min}}}$ attached at
any minimum $\alpha_{\mathrm{min}}$ and we can do the same
construction at any $\alpha_{\mathrm{min}}$.

A construction of approximate eigenfunctions of the operator $H^h_D$
in $D=\overline{U}$ is given in the next theorem.

\begin{theorem}\label{YK:l1}
For any critical point $\alpha_{\mathrm{min}}$ and
for  any $\lambda$ in the spectrum of $K=K_{\alpha_{\mathrm{min}}}$, there exist $C_1>0$,
$h_1>0$, and a family $U^h\in C^\infty_c(D)$, $h\in (0,h_1]$, such
that, for any $h\in (0,h_1]$, we have
\[
\|\left(H^{h}_D-z(h)\right)U^h\|_{L^2}\leq C_1
h^{\frac{4k+7}{2k+4}}\|U^h\|_{L^2}\,,
\]
with
\[
z(h)= \hat{\nu} \omega_{\mathrm{min}}(B)^{\frac{2}{k+2}}
h^{\frac{2k+2}{k+2}}+ \lambda h^{\frac{2k+3}{k+2}}\,.
\]
\end{theorem}

\begin{proof}
 The proof of this theorem is long and will be divided into several
steps.
\end{proof}

\subsection{Formal expansions near the minimum}
Choose a function $\phi\in C^\infty(B(0,r))$ such that
$d\phi=\omega_{0,0}$. For some $\alpha\in \RR$, we make a gauge
transformation
\begin{multline*}
u(t,s)=\exp\left(-i\frac{\phi(s)}{h} \right)\exp\left( i\frac{\alpha
\sum\limits_{j=1}^{n-1}\omega_{0,1}^{(j)}(0)s_j}{\omega_{\mathrm{min}}(B)^{\frac{k+1}{k+2}}
h^{\frac{1}{k+2}}}\right) v(t,s), \\ s\in B(0,r),\quad t\in \RR.
\end{multline*}
Then we have
\begin{equation}\label{e:HhPh}
H^{h}_0u(t,s) = \exp\left(-i\frac{\phi(s)}{h} \right)\exp\left(
i\frac{\alpha
\sum\limits_{j=1}^{n-1}\omega_{0,1}^{(j)}(0)s_j}{\omega_{\mathrm{min}}(B)^{\frac{k+1}{k+2}}
h^{\frac{1}{k+2}}} \right) P^{h}v(t,s),
\end{equation}
where
\[
P^{h}=P^h_1+P^h_2+P^h_3+P^h_4+P^h_5+P^h_6\,,
\]
and
\begin{align*}
P^h_1 = & - h^2
 \frac{{\partial}^2}{\partial t^2} - h^2 \dot{g}^{00}(s) t \frac{{\partial}^2}{\partial
 t^2},\\
P^h_2= & 2 i h \sum_{1\leq j\leq n-1} \dot{g}^{0j}(s) t \\ & \times
\left(i h \frac{\partial}{\partial s_j}+
\frac{1}{k+1}\omega_{0,1}^{(j)}(s)t^{k+1}-\alpha
\omega_{\mathrm{min}}(B)^{-\frac{k+1}{k+2}}
h^{\frac{k+1}{k+2}}\omega_{0,1}^{(j)}(0)\right) \frac{\partial
}{\partial t},\\
 P^h_3= & i h \sum_{1\leq j\leq n-1} \dot{g}^{0j}(s)
\omega_{0,1}^{(j)}(s)t^{k+1},\\
P^h_4 = & \sum_{1\leq j,\ell\leq n-1} G^{j\ell}(s) \Big(i h
\frac{\partial}{\partial s_j}+
\frac{1}{k+1}\omega_{0,1}^{(j)}(s)t^{k+1} -\alpha
\omega_{\mathrm{min}}(B)^{-\frac{k+1}{k+2}}
h^{\frac{k+1}{k+2}}\omega_{0,1}^{(j)}(0)\\ &
+\frac{1}{k+2}\omega_{0,2}^{(j)}(s)t^{k+2}\Big) \Big(i h
\frac{\partial}{\partial s_\ell}+
\frac{1}{k+1}\omega_{0,1}^{(\ell)}(s)t^{k+1}\\ & -\alpha
\omega_{\mathrm{min}}(B)^{-\frac{k+1}{k+2}}
h^{\frac{k+1}{k+2}}\omega_{0,1}^{(j)}(0)
+\frac{1}{k+2}\omega_{0,2}^{(j)}(s)t^{k+2}\Big),\\
 P^h_5 = &
\sum_{1\leq j,\ell\leq n-1} \dot{g}^{j\ell}(s)t \Big(i h
\frac{\partial}{\partial s_j}+
\frac{1}{k+1}\omega_{0,1}^{(j)}(s)t^{k+1} \\ & -\alpha
\omega_{\mathrm{min}}(B)^{-\frac{k+1}{k+2}}
h^{\frac{k+1}{k+2}}\omega_{0,1}^{(j)}(0)\Big) \Big(i h
\frac{\partial}{\partial s_\ell}+
\frac{1}{k+1}\omega_{0,1}^{(\ell)}(s)t^{k+1}\\ & -\alpha
\omega_{\mathrm{min}}(B)^{-\frac{k+1}{k+2}}
h^{\frac{k+1}{k+2}}\omega_{0,1}^{(j)}(0)\Big),\\
P^h_6= & - h^2\Gamma^0_0(s) \frac{\partial }{\partial t}  - h
\sum_{1\leq j\leq n-1}\Gamma^j_0(s) \Big(i h
\frac{\partial}{\partial s_j}
+\frac{1}{k+1}\omega_{0,1}^{(j)}(s)t^{k+1}\\ & -\alpha
\omega_{\mathrm{min}}(B)^{-\frac{k+1}{k+2}}
h^{\frac{k+1}{k+2}}\omega_{0,1}^{(j)}(0)\Big)\,.
\end{align*}
We now make the change of variables
\begin{equation}\label{e:change}
t=\omega_{\mathrm{min}}(B)^{-\frac{1}{k+2}} h^{1/(k+2)}\tau, \quad
s=h^{1/2(k+2)}\sigma\,,
\end{equation}
and expand the operators in powers of $h$ as $h\to 0$ up to the
terms of order $O(h^{(4k+7)/(2(k+2))})$. For the first three terms,
we obtain
\begin{align*}
\widehat P_1^h= & - h^{\frac{2k+2}{k+2}}
\omega_{\mathrm{min}}(B)^{\frac{2}{k+2}}
\frac{{\partial}^2}{\partial \tau^2} - h^{\frac{2k+3}{k+2}}
\omega_{\mathrm{min}}(B)^{\frac{1}{k+2}}  \dot{g}^{00}(s) \tau
\frac{{\partial}^2}{\partial \tau^2}+ O(h^{\frac{4k+7}{2(k+2)}}),\\
\widehat P_2^h = & 2 i h^{\frac{2k+3}{k+2}} \sum_{1\leq j\leq n-1}
\dot{g}^{0j}(0) \omega_{\mathrm{min}}(B)^{-\frac{k+1}{k+2}}
\omega_{0,1}^{(j)}(0) \tau
\left(\frac{\tau^{k+1}}{k+1}-\alpha\right) \frac{\partial }{\partial
\tau}+ O(h^{\frac{4k+7}{2(k+2)}}),\\
\widehat P_3^h  = & i h^{\frac{2k+3}{k+2}} \sum_{1\leq j\leq n-1}
\dot{g}^{0j}(0)
\omega_{0,1}^{(j)}(0)\omega_{\mathrm{min}}(B)^{-\frac{k+1}{k+2}}
\tau^{k+1}+ O(h^{\frac{4k+7}{2(k+2)}})\,.
\end{align*}
For the analysis of the fourth term, let us start with the
computation of
\begin{align*} I = & \Big(i h \frac{\partial}{\partial s_j}+
\frac{1}{k+1}\omega_{0,1}^{(j)}(s)t^{k+1}-\alpha
\omega_{\mathrm{min}}(B)^{-\frac{k+1}{k+2}}
h^{\frac{k+1}{k+2}}\omega_{0,1}^{(j)}(0)\\
& +\frac{1}{k+2}\omega_{0,2}^{(j)}(s)t^{k+2}\Big)  \\ & \times
\Big(i h \frac{\partial}{\partial s_\ell}+
\frac{1}{k+1}\omega_{0,1}^{(\ell)}(s)t^{k+1}-\alpha
\omega_{\mathrm{min}}(B)^{-\frac{k+1}{k+2}}
h^{\frac{k+1}{k+2}}\omega_{0,1}^{(j)}(0)\\ &
+\frac{1}{k+2}\omega_{0,2}^{(j)}(s)t^{k+2}\Big).
\end{align*}
After the change of variables, we obtain
\begin{align*}
{\widehat I} = & \omega_{\mathrm{min}}(B)^{-\frac{2k+2}{k+2}}
h^{\frac{2k+2}{k+2}}\left(\frac{\tau^{k+1}}{k+1}\omega_{0,1}^{(j)}(h^{1/2(k+2)}\sigma)-\alpha
\omega_{0,1}^{(j)}(0)\right) \\ & \times
\left(\frac{\tau^{k+1}}{k+1}\omega_{0,1}^{(\ell)}(h^{1/2(k+2)}\sigma)-\alpha
\omega_{0,1}^{(\ell)}(0)\right)\\
& + i \omega_{\mathrm{min}}(B)^{-\frac{k+1}{k+2}}
h^{\frac{2k+2}{k+2}+\frac{1}{2(k+2)}} \Big[ \frac{\partial}{\partial
\sigma_j}
\left(\frac{\tau^{k+1}}{k+1}\omega_{0,1}^{(\ell)}(h^{1/2(k+2)}\sigma)-\alpha
\omega_{0,1}^{(\ell)}(0)\right) \\
& + \left(\frac{\tau^{k+1}}{k+1}
\omega_{0,1}^{(j)}(h^{1/2(k+2)}\sigma)-\alpha
\omega_{0,1}^{(j)}(0)\right) \frac{\partial}{\partial
\sigma_\ell}\Big]\\
& + h^{\frac{2k+3}{k+2}} \Big[ - \frac{\partial^2}{\partial
\sigma_j\partial \sigma_\ell}
+\omega_{\mathrm{min}}(B)^{-\frac{2k+3}{k+2}} \frac{\tau^{k+2}}{k+2}
\\ & \times
\Big(\left(\frac{\tau^{k+1}}{k+1}\omega_{0,1}^{(j)}(h^{1/2(k+2)}\sigma)-\alpha
\omega_{0,1}^{(j)}(0)\right)\omega_{0,2}^{(\ell)}(h^{1/2(k+2)}\sigma)
 \\
& +
\left(\frac{\tau^{k+1}}{k+1}\omega_{0,1}^{(\ell)}(h^{1/2(k+2)}\sigma)-\alpha
\omega_{0,1}^{(\ell)}(0)\right)\omega_{0,2}^{(j)}(h^{1/2(k+2)}\sigma)\Big)\Big]\,.
\end{align*}

We can write
\begin{multline}\label{e:exp-omega}
\omega_{0,1}^{(j)}(h^{1/2(k+2)}\sigma)=\omega_{0,1}^{(j)}(0)+h^{\frac{1}{2(k+2)}}
\sum \frac{\partial \omega_{0,1}^{(j)}}{\partial s_r}(0)\sigma_r
\\ +\frac{1}{2}h^{\frac{1}{k+2}}\sum \frac{\partial^2
\omega_{0,1}^{(j)}}{\partial s_r \partial s_m}(0)\sigma_r \sigma_m
+O(h^{\frac{3}{2(k+2)}}), \quad h\to 0\,.
\end{multline}
Since $s=0$ is a minimum of the function
\begin{equation}\label{e:omega2}
|\omega_{0,1}(s)|^2=\sum_{j,\ell}G^{j\ell}(s) \omega_{0,1}^{(j)}(s)
\omega_{0,1}^{(\ell)}(s)\,,
\end{equation}
we have
\begin{equation}\label{e:domega2}
\left( \frac{\partial }{\partial s_r}|\omega_{0,1}|^2
\right)(0)=2\sum_{j} \frac{\partial \omega_{0,1}^{(j)}}{\partial
s_r}(0)\omega_{0,1}^{(j)}(0)=0\,.
\end{equation}
Using \eqref{e:exp-omega}, we obtain that
\begin{multline*}
\left(\frac{\tau^{k+1}}{k+1}\omega_{0,1}^{(j)}(h^{1/2(k+2)}\sigma)-\alpha
\omega_{0,1}^{(j)}(0)\right)
\left(\frac{\tau^{k+1}}{k+1}\omega_{0,1}^{(\ell)}(h^{1/2(k+2)}\sigma)-\alpha
\omega_{0,1}^{(\ell)}(0)\right) \\
\begin{aligned}
= &
\omega_{0,1}^{(j)}(0)\omega_{0,1}^{(\ell)}(0)\left(\frac{\tau^{k+1}}{k+1}-\alpha
\right)^2  +  h^{\frac{1}{2(k+2)}} \Big( \omega_{0,1}^{(\ell)}(0)
\sum \frac{\partial \omega_{0,1}^{(j)}}{\partial s_r}(0)\sigma_r
\\
& + \omega_{0,1}^{(j)}(0)  \sum \frac{\partial
\omega_{0,1}^{(\ell)}}{\partial s_r}(0)\sigma_r \Big)
\frac{\tau^{k+1}}{k+1} \left(\frac{\tau^{k+1}}{k+1}-\alpha\right)\\
& +  h^{\frac{1}{k+2}} \Big[\frac{1}{2}  \Big(
\omega_{0,1}^{(\ell)}(0) \sum \frac{\partial^2
\omega_{0,1}^{(j)}}{\partial s_r
\partial s_m}(0)\sigma_r \sigma_m\\
& + \omega_{0,1}^{(j)}(0)   \sum \frac{\partial^2
\omega_{0,1}^{(\ell)}}{\partial s_r \partial s_m}(0)\sigma_r
\sigma_m \Big) \frac{\tau^{k+1}}{k+1}
\left(\frac{\tau^{k+1}}{k+1}-\alpha\right) \\
& + \sum_{r,m} \frac{\partial \omega_{0,1}^{(j)}}{\partial s_r}(0)
\frac{\partial \omega_{0,1}^{(\ell)}}{\partial s_m}(0)\sigma_r
\sigma_m \frac{\tau^{2k+2}}{(k+1)^2}\Big]+ O(h^{\frac{3}{2(k+2)}})\,.
\end{aligned}
\end{multline*}
Similarly, we have
\begin{multline*}
\frac{\partial}{\partial \sigma_j}
\left(\frac{\tau^{k+1}}{k+1}\omega_{0,1}^{(\ell)}(h^{1/2(k+2)}\sigma)-\alpha
\omega_{0,1}^{(\ell)}(0)\right)\\ +
\left(\frac{\tau^{k+1}}{k+1}\omega_{0,1}^{(j)}(h^{1/2(k+2)}\sigma)-\alpha
\omega_{0,1}^{(j)}(0)\right) \frac{\partial}{\partial
\sigma_\ell}\\
= \Big[\omega_{0,1}^{(j)}(0) \frac{\partial}{\partial \sigma_\ell}+
\omega_{0,1}^{(\ell)}(0)
\frac{\partial}{\partial \sigma_j}\Big]\left(\frac{\tau^{k+1}}{k+1}-\alpha\right)\\
+ h^{\frac{1}{2(k+2)}} \left[\frac{\partial
\omega_{0,1}^{(\ell)}}{\partial \sigma_j}(0) +\sum \frac{\partial
\omega_{0,1}^{(\ell)}}{\partial \sigma_r}(0)\sigma_r\right]
\frac{\tau^{k+1} }{k+1} + O(h^{\frac{1}{k+2}})\,.
\end{multline*}
Thus, we get the following expression for ${\widehat I}$:
\begin{align*}
{\widehat I}= & \omega_{\mathrm{min}}(B)^{-\frac{2k+2}{k+2}}
h^{\frac{2k+2}{k+2}}\omega_{0,1}^{(j)}(0)
\omega_{0,1}^{(\ell)}(0)\left(\frac{\tau^{k+1}}{k+1}-\alpha
\right)^2  \\
& + h^{\frac{2k+2}{k+2}+\frac{1}{2(k+2)}}\Big[ i
\omega_{\mathrm{min}}(B)^{-\frac{k+1}{k+2}}
\Big[\omega_{0,1}^{(j)}(0) \frac{\partial}{\partial \sigma_\ell}+
\omega_{0,1}^{(\ell)}(0)
\frac{\partial}{\partial \sigma_j}\Big]\left(\frac{\tau^{k+1}}{k+1}-\alpha\right) \\
& + \omega_{\mathrm{min}}(B)^{-\frac{2k+2}{k+2}}\Big(
\omega_{0,1}^{(\ell)}(0) \sum \frac{\partial
\omega_{0,1}^{(j)}}{\partial s_r}(0)\sigma_r\\
& + \omega_{0,1}^{(j)}(0)  \sum \frac{\partial
\omega_{0,1}^{(\ell)}}{\partial s_r}(0)\sigma_r
\Big)\frac{\tau^{k+1}}{k+1}
\left(\frac{\tau^{k+1}}{k+1}-\alpha\right)\Big]\\ & +
h^{\frac{2k+3}{k+2}} \Bigg[ - \frac{\partial^2}{\partial
\sigma_j\partial \sigma_\ell} + i
\omega_{\mathrm{min}}(B)^{-\frac{k+1}{k+2}}  \frac{\partial
\omega_{0,1}^{(\ell)}}{\partial \sigma_j}(0)  \frac{\tau^{k+1} }{k+1} \\
& +\omega_{\mathrm{min}}(B)^{-\frac{2k+3}{k+2}}
\left(\omega_{0,1}^{(j)}(0)\omega_{0,2}^{(\ell)}(0)+\omega_{0,1}^{(\ell)}(0)\omega_{0,2}^{(j)}(0)
\right) \frac{\tau^{k+2}}{k+2} \Big(\frac{\tau^{k+1}}{k+1}-\alpha
\Big)\\
& +  \omega_{\mathrm{min}}(B)^{-\frac{2k+2}{k+2}}\Big[\frac{1}{2}
\Big( \omega_{0,1}^{(\ell)}(0) \sum \frac{\partial^2
\omega_{0,1}^{(j)}}{\partial s_r
\partial s_m}(0)\sigma_r \sigma_m
 \\ & + \omega_{0,1}^{(j)}(0)   \sum
\frac{\partial^2 \omega_{0,1}^{(\ell)}}{\partial s_r \partial
s_m}(0)\sigma_r \sigma_m \Big) \frac{\tau^{k+1}}{k+1}
\left(\frac{\tau^{k+1}}{k+1}-\alpha\right) \\
& + \sum_{r,m} \frac{\partial \omega_{0,1}^{(j)}}{\partial s_r}(0)
\frac{\partial \omega_{0,1}^{(\ell)}}{\partial s_m}(0)\sigma_r
\sigma_m \frac{\tau^{2k+2}}{(k+1)^2}\Big]\Bigg] +
O(h^{\frac{4k+4}{2(k+2)}}).
\end{align*}
Using \eqref{e:domega2} and the expansion
\[
G^{j\ell}(h^{1/2(k+2)}\sigma) \\ =\delta^{j\ell}+
\frac{1}{2}h^{\frac{1}{k+2}}\sum \frac{\partial^2
G^{j\ell}}{\partial s_r \partial s_m}(0)\sigma_r \sigma_m
+O(h^{\frac{3}{2(k+2)}}), \quad h\to 0\,,
\]
we obtain that, after the change of variables (\ref{e:change}), the
fourth term takes the form:
\begin{align*}
\widehat P_4^h =   & \omega_{\mathrm{min}}(B)^{\frac{2}{k+2}}
h^{\frac{2k+2}{k+2}} \left(\frac{\tau^{k+1}}{k+1}-\alpha
\right)^2  \\
& + h^{\frac{2k+2}{k+2}+\frac{1}{2(k+2)}} 2i
\omega_{\mathrm{min}}(B)^{-\frac{k+1}{k+2}}
\left(\frac{\tau^{k+1}}{k+1}-\alpha\right) \sum
\omega_{0,1}^{(j)}(0) \frac{\partial}{\partial \sigma_j}   \\ & +
h^{\frac{2k+3}{k+2}} \Bigg[ - \sum \frac{\partial^2}{\partial
\sigma_j^2} + i \omega_{\mathrm{min}}(B)^{-\frac{k+1}{k+2}} \sum_j
\frac{\partial \omega_{0,1}^{(j)}}{\partial
\sigma_j}(0) \frac{\tau^{k+1} }{k+1} \\
& +2\omega_{\mathrm{min}}(B)^{-\frac{2k+3}{k+2}} \sum
\omega_{0,1}^{(j)}(0)\omega_{0,2}^{(j)}(0) \frac{\tau^{k+2}}{k+2}
\Big(\frac{\tau^{k+1}}{k+1}-\alpha
\Big)\\
& +  \omega_{\mathrm{min}}(B)^{-\frac{2k+2}{k+2}}\Big[\sum_{r,m}
\Big(\sum_j \omega_{0,1}^{(j)}(0) \frac{\partial^2
\omega_{0,1}^{(j)}}{\partial s_r
\partial s_m}(0)\Big)\sigma_r \sigma_m \frac{\tau^{k+1}}{k+1}
\left(\frac{\tau^{k+1}}{k+1}-\alpha\right) \\
& +  \sum_{r,m} \left( \sum_j \frac{\partial
\omega_{0,1}^{(j)}}{\partial s_r}(0) \frac{\partial
\omega_{0,1}^{(j)}}{\partial s_m}(0)\right)\sigma_r \sigma_m
\frac{\tau^{2k+2}}{(k+1)^2}
\\
& + \frac{1}{2}\sum_{rm} \left(\sum_{j\ell} \frac{\partial^2
G^{j\ell}}{\partial s_r \partial s_m}(0) \omega_{0,1}^{(j)}(0)
\omega_{0,1}^{(\ell)}(0)\right) \sigma_r \sigma_m
\left(\frac{\tau^{k+1}}{k+1}-\alpha \right)^2 \Big] \Bigg]\\
& + O(h^{\frac{4k+4}{2(k+2)}}).
\end{align*}
Finally, after the change of variables (\ref{e:change}), the fifth
and sixth terms become:
\begin{multline*}
\widehat P_5^h = \omega_{\mathrm{min}}(B)^{-\frac{2k+3}{k+2}}
h^{\frac{2k+3}{k+2}}\sum_{1\leq j,\ell\leq n-1} \dot{g}^{j\ell}(s)
\omega_{0,1}^{(j)}(0) \omega_{0,1}^{(\ell)}(0) \tau
\left(\frac{\tau^{k+1}}{k+1}-\alpha \right)^2\\ +
O(h^{\frac{4k+4}{2(k+2)}}),
\end{multline*}
\begin{multline*}
 \widehat P_6^h= - \Gamma^0_0(0)
\omega_{\mathrm{min}}(B)^{\frac{1}{k+2}} h^{\frac{2k+3}{k+2}}
\frac{\partial }{\partial \tau} \\ - \sum_{1\leq j\leq n-1}
h^{\frac{2k+3}{k+2}} \Gamma^j_0(0)
\omega_{\mathrm{min}}(B)^{-\frac{k+1}{k+2}} \omega_{0,1}^{(j)}(0) (
\frac{\tau^{k+1}}{k+1} -\alpha)+ O(h^{\frac{4k+4}{2(k+2)}}).
\end{multline*}

Thus, after the change of variables (\ref{e:change}), the operator
$P^h$ has a formal asymptotic expansion
\[
{\widehat P}^{h}=\omega_{\mathrm{min}}(B)^{\frac{2}{k+2}}
h^{\frac{2k+2}{k+2}}\sum_{\ell=0}^{\infty}
h^{\frac{1}{2(k+2)}}\widehat P_\ell,
\]
where
\begin{equation}\label{e:defh10}
\widehat P_0 = -\frac{\partial^2 }{\partial \tau^2}+
\left(\frac{\tau^{k+1}}{k+1} - \alpha \right)^2=Q(\alpha,1),
\end{equation}
\begin{equation}\label{e:defh11}
\widehat P_1 = 2i \omega_{\mathrm{min}}(B)^{-\frac{k+3}{k+2}}
\left(\frac{\tau^{k+1}}{k+1}-\alpha\right) \sum
\omega_{0,1}^{(j)}(0) \frac{\partial}{\partial \sigma_j},
\end{equation}
and
\begin{equation}\label{e:defh12}
\begin{split}
\widehat P_2 = & -\dot{g}^{00}(0)
\omega_{\mathrm{min}}(B)^{-\frac{1}{k+2}}\tau
\frac{{\partial}^2}{\partial
 \tau^2} \\ & +
2 i \sum_{1\leq j\leq n-1} \dot{g}^{0j}(0)
\omega_{\mathrm{min}}(B)^{-\frac{k+3}{k+2}} \omega_{0,1}^{(j)}(0)
\tau \left(\frac{\tau^{k+1}}{k+1}-\alpha\right) \frac{\partial
}{\partial \tau}\\ & + i \sum_{1\leq j\leq n-1} \dot{g}^{0j}(0)
\omega_{0,1}^{(j)}(0)\omega_{\mathrm{min}}(B)^{-\frac{k+3}{k+2}}
\tau^{k+1} \\
& - \omega_{\mathrm{min}}(B)^{-\frac{2}{k+2}} \sum
\frac{\partial^2}{\partial \sigma_j^2} + i
\omega_{\mathrm{min}}(B)^{-\frac{k+3}{k+2}} \sum_j \frac{\partial
\omega_{0,1}^{(j)}}{\partial
\sigma_j}(0) \frac{\tau^{k+1} }{k+1} \\
& +2\omega_{\mathrm{min}}(B)^{-\frac{2k+5}{k+2}} \sum
\omega_{0,1}^{(j)}(0)\omega_{0,2}^{(j)}(0) \frac{\tau^{k+2}}{k+2}
\Big(\frac{\tau^{k+1}}{k+1}-\alpha
\Big)\\
& +  \omega_{\mathrm{min}}(B)^{-2}\Big[\sum_{r,m} \Big(\sum_j
\omega_{0,1}^{(j)}(0) \frac{\partial^2 \omega_{0,1}^{(j)}}{\partial
s_r
\partial s_m}(0)\Big)\sigma_r \sigma_m \frac{\tau^{k+1}}{k+1}
\left(\frac{\tau^{k+1}}{k+1}-\alpha\right) \\
& + \sum_{r,m} \left( \sum_j \frac{\partial
\omega_{0,1}^{(j)}}{\partial s_r}(0) \frac{\partial
\omega_{0,1}^{(j)}}{\partial s_m}(0)\right)\sigma_r \sigma_m
\frac{\tau^{2k+2}}{(k+1)^2}\\ & + \frac{1}{2}\sum_{r,m}
\left(\sum_{j,\ell} \frac{\partial^2 G^{j\ell}}{\partial s_r
\partial s_m}(0) \omega_{0,1}^{(j)}(0)
\omega_{0,1}^{(\ell)}(0)\right) \sigma_r \sigma_m
\left(\frac{\tau^{k+1}}{k+1}-\alpha \right)^2\Big]\\
& + \omega_{\mathrm{min}}(B)^{-\frac{2k+5}{k+2}} \sum_{1\leq
j,\ell\leq n-1} \dot{g}^{j\ell}(0) \omega_{0,1}^{(j)}(0)
\omega_{0,1}^{(\ell)}(0) \tau
\left(\frac{\tau^{k+1}}{k+1}-\alpha \right)^2\\
& - \Gamma^0_0(0) \omega_{\mathrm{min}}(B)^{-\frac{1}{k+2}}
\frac{\partial }{\partial \tau} \\ & -
\omega_{\mathrm{min}}(B)^{-\frac{k+3}{k+2}} \sum_{1\leq j\leq
n-1}\Gamma^j_0(0)\omega_{0,1}^{(j)}(0) \left( \frac{\tau^{k+1}}{k+1}
-\alpha\right).
\end{split}
\end{equation}

\subsection{Reduction to the zero set}
Now we use the method initiated by Grushin \cite{Grushin} (and
references therein) and Sj\"ostrand \cite{Sj73} in the context of
hypoellipticity. We will closely follow the exposition in
\cite{Fournais-Helffer06} (see also \cite{syrievienne}). We now choose
 some $\alpha_{min}$ and will use the previous construction at
\[
\alpha=\alpha_{\rm min}\,.
\]
 The starting point is to consider the
operator in $\cS(\RR^n)\times \cS(\RR^{n-1})$ defined by
\begin{equation*}
    \cP_0=\begin{pmatrix}
            P_0 & R^-_0 \\
            R^+_0 & 0 \\
          \end{pmatrix},
\end{equation*}
where the operator $P_0 : \cS(\RR^n) \to \cS(\RR^{n})$ is given by
\[
P_0=-\frac{\partial^2 }{\partial \tau^2}+
\left(\frac{\tau^{k+1}}{k+1} - \alpha_{\rm min} \right)^2
-\hat{\nu}=Q(\alpha_{\rm min},1)-\hat{\nu},
\]
the operator $R^-_0 : \cS(\RR^{n-1}) \to \cS(\RR^{n})$ is given by
\[
R^-_0\phi(\tau,\sigma)= \phi(\sigma)
u^0_{\alpha_{\mathrm{min}}}(\tau), \quad \phi\in \cS(\RR^{n-1}),
\]
and the operator $R^+_0 : \cS(\RR^{n}) \to \cS(\RR^{n-1})$ is given
by
\[
R^+_0f(\sigma)= \int f(\tau,\sigma)
u^0_{\alpha_{\mathrm{min}}}(\tau) d\tau, \quad f\in \cS(\RR^{n}).
\]
We observe that $\cP_0$ considered as an operator in
$L^2(\RR^n,d\tau\,d\sigma)\times L^2(\RR^{n-1},d\sigma)$ is
formally self-adjoint. In particular $R^-_0$ is the Hilbertian
adjoint of $R^+_0$ (considered as an operator from
$L^2(\RR^{n},d\tau\,d\sigma)$ to $L^2(\RR^{n-1},d\sigma)$).

We also verify that
\[
R^+_0R^-_0= I_{L^2(\RR^{n-1})}, \quad R^-_0R^+_0=\Pi_0,
\]
where $\Pi_0 : L^2(\RR^{n}) \to L^2(\RR^{n})$ is the orthogonal
projection on the subspace $\{\RR
u^0_{\alpha_{\mathrm{min}}}\}\times L^2(\RR^{n-1})$ in
$L^2(\RR^{n})$:
\[
\Pi_0f(\tau,\sigma)= \left(\int f(\tau,\sigma)
u^0_{\alpha_{\mathrm{min}}}(\tau) d\tau \right)
u^0_{\alpha_{\mathrm{min}}}(\tau), \quad f\in L^2(\RR^{n}).
\]

Define the operator $E_0$ in $\cS(\RR^{n})$ by
\begin{equation}\label{e:defE0}
E_0= (I-\Pi_0)P_0^{-1} (I-\Pi_0),
\end{equation}
where, by abuse of notation, we consider $P_0$ as an operator in
$L^2(\RR^{n})$. As shown in \cite[Lemma A.5]{Fournais-Helffer06},
$E_0$ respects the Schwartz space $\cS(\RR^n)$. Then we have
\[
\cP_0\circ \cE_0=I\,,
\]
where the operator $\cE_0$ in $\cS(\RR^n)\times \cS(\RR^{n-1})$ is
given by the matrix
\begin{equation*}
    \cE_0=\begin{pmatrix}
            E_0 & R^-_0 \\
            R^+_0 & 0 \\
          \end{pmatrix}.
\end{equation*}

The idea is to consider the more general operator $\cP(z)$ in
$\cS(\RR^n)\times \cS(\RR^{n-1})$ defined by
\begin{equation*}
    \cP(z)=\begin{pmatrix}
            \omega_{\mathrm{min}}(B)^{-\frac{2}{k+2}}
h^{-\frac{2k+2}{k+2}}({\widehat P}^{h}- z) &\quad\quad R^-_0 \\
            R^+_0 & \quad\quad 0 \\
          \end{pmatrix}.
\end{equation*}
Note that $\cP(z)$ is for $z\in \RR$ formally self-adjoint for the
original $L^2$-scalar product in $L^2(\RR^n)$ but not for the usual
$L^2$ associated to the standard Lebesgue measure $d\tau\,d\sigma)$.

We are looking for the right inverse of $\cP(z)$ for any small $z\in
\CC$. One can write
\begin{equation*}
    \cP(z)=\begin{pmatrix}
P_0+ \delta P -Z &\quad\quad  R^-_0 \\
            R^+_0 &\quad\quad  0 \\
          \end{pmatrix}\,,
\end{equation*}
where
\begin{align}
\delta P& =\omega_{\mathrm{min}}(B)^{-\frac{2}{k+2}}
h^{-\frac{2k+2}{k+2}}{\widehat P}^{h}-P_0\\
& = -\hat{\nu}+ h^{\frac{1}{2(k+2)}}\widehat P_1
+h^{\frac{1}{k+2}}\widehat P_2+h^{\frac{3}{2(k+2)}}Q(h), \nonumber\\
Z& =\omega_{\mathrm{min}}(B)^{-\frac{2}{k+2}}
h^{-\frac{2k+2}{k+2}}z-\hat{\nu}\,,\label{defZ}
\end{align}
and $Q(h)$ admits a complete expansion
\[
Q(h)\sim \sum_{j=0}^\infty h^{\frac{j}{2(k+2)}} \widehat P_{j+3}\,.
\]

We first observe that
\[
\cP(z)\cE_0 = I+\cK,
\]
where
\[
\cK=\begin{pmatrix}
            (\delta P -Z) E_0 &\quad  (\delta P -Z) R_0^+ \\
            0 & \quad 0 \\
          \end{pmatrix}.
\]

We will assume that $Z$ is a function of $h$, $Z=Z(h)$, which admits
a formal asymptotic expansion of the form
\begin{equation}\label{e:defZ}
Z(h)\sim \sum_{\ell\geq 1} Z_\ell h^{\frac{\ell}{2(k+2)}}.
\end{equation}
Then we have
\[
\delta P -Z\sim \sum_{\ell\geq 1} (\widehat P_\ell-Z_\ell)
h^{\frac{\ell}{2(k+2)}}\,.
\]
If we define
\[
Q\sim \sum_{j=0}^{+\infty}(-1)^j\cK^j\,,
\]
then the operator is well-defined (after reordering) as a formal
expansion in powers of $h^{\frac{1}{2(k+2)}}$ and
\[
\cP(z)\cE_0 Q\sim I\,.
\]
So $\cE(z)=\cE_0 Q$ is the right inverse of $\cP(z)$. If we write
\begin{equation*}
    \cE(z)=\begin{pmatrix}
            E(z) & E^+(z) \\
            E^-(z) & E^{\pm}(z) \\
          \end{pmatrix},
\end{equation*}
we get, in the sense of formal expansions in powers of
$h^{\frac{1}{2(k+2)}}$,
\begin{align}
\omega_{\mathrm{min}}(B)^{-\frac{2}{k+2}}
h^{-\frac{2k+2}{k+2}}({\widehat P}^{h}- z) E(z) + R^-_0E^-(z)&\sim I, \\
\omega_{\mathrm{min}}(B)^{-\frac{2}{k+2}}
h^{-\frac{2k+2}{k+2}}({\widehat P}^{h}- z) E^+(z) + R^-_0E^{\pm}(z)&\sim 0. \label{e:Re1}\\
R^+_0 E(z) & \sim 0. \\
R^+_0 E^+(z) & \sim I.
\end{align}

Let us introduce a function $\widehat E^{\pm}(Z)$ by
$$
\widehat E^{\pm} (Z) = E^{\pm} (z)\,,
$$
with $Z$ related to $z$ by \eqref{defZ}. We have
\[
\cK^j=\begin{pmatrix}
            [(\delta P -Z) E_0]^j & [(\delta P -Z) E_0]^{j-1} [(\delta P -Z) R_0^-] \\
            0 & 0 \\
          \end{pmatrix},
\]
and therefore, $\widehat E^{\pm}(Z)$ is the following asymptotic
series,
\begin{equation}\label{e:Epm}
\widehat E^{\pm}(Z)\sim \sum_{j=1}^{+\infty} (-1)^j R_0^+[(\delta P
-Z) E_0]^{j-1} [(\delta P -Z) R_0^-].
\end{equation}

\begin{lemma}\label{l:expEpm}
The function $\widehat E^\pm(Z(h))$ admits the following formal
asymptotic expansion in powers of $h^{\frac{1}{2(k+2)}}$:
\[
\widehat E^\pm(Z(h))=\sum_{j=1}^\infty
E^{\pm}_{\frac{j}{2(k+2)}}h^{\frac{j}{2(k+2)}},
\]
with
\begin{align}
E^{\pm}_{\frac{1}{2(k+2)}}& =Z_1, \label{z1} \\
E^{\pm}_{\frac{1}{k+2}}&
=Z_2-\omega_{\mathrm{min}}(B)^{-\frac{2}{k+2}} K\,. \label{z2}
\end{align}
\end{lemma}

\begin{proof}
It follows from \eqref{e:Epm} that the coefficient of
$h^{\frac{1}{2(k+2)}}$ is given by
\[
E^{\pm}_{\frac{1}{2(k+2)}} = Z_1-R_0^+\widehat P_1R_0^-.
\]
Using the definitions of the operators $R_0^+$, $\widehat P_1$ and
$R_0^-$ and (\ref{e:lambda01}), we get
\begin{align*}
E^{\pm}_{\frac{1}{2(k+2)}} = & Z_1-
2i\omega_{\mathrm{min}}(B)^{-\frac{k+3}{k+2}} \left[\int
\left(\frac{\tau^{k+1}}{k+1}-\alpha_{\rm
min}\right)|u^0_{\alpha_{\mathrm{min}}}(\tau)|^2 \,d\tau \right]
\sum_{j}
\omega_{0,1}^{(j)}(0)\frac{\partial}{\partial \sigma_j} \\
 = & Z_1.
\end{align*}
By \eqref{e:Epm}, the coefficient of $h^{\frac{1}{k+2}}$ is given by
the operator
\begin{equation}\label{def1/k+2}
E^{\pm}_{\frac{1}{k+2}} = Z_2-R_0^+\widehat P_2R_0^- + R_0^+
\widehat P_1 E_0 \widehat P_1 R_0^-.
\end{equation}
By \eqref{e:defh12}, it follows that
\[
R_0^+\widehat P_2R_0^-= \omega_{\mathrm{min}}(B)^{-\frac{2}{k+2}}
\sum_{j} \frac{\partial^2}{\partial \sigma_j^2} + \sum_{r,m} b_{rm}
\sigma_r \sigma_m +a,
\]
where $b_{rm}$ and $a$ are given by
\begin{multline*}
b_{rm}=- \omega_{\mathrm{min}}(B)^{-2}\Big[ R_0^+
\frac{\tau^{k+1}}{k+1} \left(\frac{\tau^{k+1}}{k+1}-\alpha_{\rm
min}\right)R_0^-\Big(\sum_j \omega_{0,1}^{(j)}(0) \frac{\partial^2
\omega_{0,1}^{(j)}}{\partial s_r
\partial s_m}(0)\Big)  \\
+ R_0^+ \frac{\tau^{2k+2}}{(k+1)^2}R_0^-  \left( \sum_j
\frac{\partial \omega_{0,1}^{(j)}}{\partial s_r}(0) \frac{\partial
\omega_{0,1}^{(j)}}{\partial s_m}(0)\right) \\
+ \frac{1}{2}R_0^+ \left(\frac{\tau^{k+1}}{k+1}-\alpha_{\rm min}
\right)^2 R_0^- \left(\sum_{j,\ell} \frac{\partial^2
G^{j\ell}}{\partial s_r \partial s_m}(0) \omega_{0,1}^{(j)}(0)
\omega_{0,1}^{(\ell)}(0)\right) \Big],
\end{multline*}
and
\begin{equation}\label{e:C}
\begin{split}
a = & -\dot{g}^{00}(0) \omega_{\mathrm{min}}(B)^{-\frac{1}{k+2}}\int
\tau \frac{{\partial}^2 u^0_\alpha}{\partial
 \tau^2}(\tau) u^0_{\alpha_{\mathrm{min}}}(\tau) d\tau \\ & +
2 i \sum_{1\leq j\leq n-1} \dot{g}^{0j}(0)
\omega_{\mathrm{min}}(B)^{-\frac{k+3}{k+2}} \omega_{0,1}^{(j)}(0)
\int \tau \left(\frac{\tau^{k+1}}{k+1}-\alpha_{\rm min}\right)
\frac{\partial u^0_\alpha}{\partial \tau}(\tau)
u^0_{\alpha_{\mathrm{min}}}(\tau) d\tau\\ & + i \sum_{1\leq j\leq
n-1} \dot{g}^{0j}(0)
\omega_{0,1}^{(j)}(0)\omega_{\mathrm{min}}(B)^{-\frac{k+3}{k+2}}
\int \tau^{k+1} (u^0_{\alpha_{\mathrm{min}}}(\tau))^2 d\tau \\
& + i \omega_{\mathrm{min}}(B)^{-\frac{k+3}{k+2}} \sum_j
\frac{\partial \omega_{0,1}^{(j)}}{\partial
\sigma_j}(0) \int \frac{\tau^{k+1} }{k+1} (u^0_{\alpha_{\mathrm{min}}}(\tau))^2 d\tau \\
& +2\omega_{\mathrm{min}}(B)^{-\frac{2k+5}{k+2}} \sum
\omega_{0,1}^{(j)}(0)\omega_{0,2}^{(j)}(0)\int
\frac{\tau^{k+2}}{k+2} \Big(\frac{\tau^{k+1}}{k+1}-\alpha_{\rm min}
\Big)(u^0_{\alpha_{\mathrm{min}}}(\tau))^2 d\tau\\
& + \omega_{\mathrm{min}}(B)^{-\frac{2k+5}{k+2}} \sum_{1\leq
j,\ell\leq n-1} \dot{g}^{j\ell}(s) \omega_{0,1}^{(j)}(0)
\omega_{0,1}^{(\ell)}(0)\int \tau
\left(\frac{\tau^{k+1}}{k+1}-\alpha_{\rm min} \right)^2(u^0_{\alpha_{\mathrm{min}}}(\tau))^2 d\tau\\
& - \Gamma^0_0(0) \omega_{\mathrm{min}}(B)^{-\frac{1}{k+2}} \int
\frac{\partial u^0_\alpha}{\partial \tau}(\tau)
u^0_{\alpha_{\mathrm{min}}}(\tau) d\tau \\ &-
\omega_{\mathrm{min}}(B)^{-\frac{k+3}{k+2}} \sum_{1\leq j\leq
n-1}\Gamma^j_0(0)\omega_{0,1}^{(j)}(0)\int \left(
\frac{\tau^{k+1}}{k+1} -\alpha_{\rm
min}\right)(u^0_{\alpha_{\mathrm{min}}}(\tau))^2 d\tau.
\end{split}
\end{equation}
By (\ref{e:lambda01}) and (\ref{e:norm}), it follows that
\begin{multline*}
R_0^+ \left(\frac{\tau^{k+1}}{k+1} -\alpha_{\rm min}\right)
\frac{\tau^{k+1}}{k+1}R_0^- \\ = \int \left(\frac{\tau^{k+1}}{k+1}
-\alpha_{\rm min}\right) \frac{\tau^{k+1}}{k+1}
(u^0_{\alpha_{\mathrm{min}}}(\tau))^2\,d\tau= \frac{\hat\nu}{k+2},
\end{multline*}
\[
R_0^+ \frac{\tau^{2k+2}}{(k+1)^2}R_0^- = \int
\frac{\tau^{2k+2}}{(k+1)^2}
u^0_{\alpha_{\mathrm{min}}}(\tau)^2\,d\tau = \frac{\hat\nu}{k+2}+
\alpha_{\rm min}^2.
\]
and
\[
R_0^+ \left(\frac{\tau^{k+1}}{k+1}-\alpha_{\rm min} \right)^2
R_0^-=\| \left(\frac{\tau^{k+1}}{k+1} -\alpha_{\rm min}\right)
u^0_{\alpha_{\mathrm{min}}}(\tau)\|^2=\frac{\hat\nu}{k+2}.
\]

Using the fact that $G_{j\ell}(0)=\delta_{j\ell}$ and ${\partial
G_{j\ell}}/{\partial s_r}(0)=0$ for any $j$, $\ell$ and $r$, the
formula~\eqref{e:omega2} implies that, for any $r$ and $m$
\begin{align*}
\frac{\partial^2 |\omega_{0,1}|^2}{\partial s_r \partial s_m}(0) = &
\sum_{j,\ell} \frac{\partial^2 G^{j\ell}}{\partial s_r
\partial s_m} (0) \omega_{0,1}^{(j)}(0)
\omega_{0,1}^{(\ell)}(0)\\
& + 2 \sum_{j}\frac{\partial^2 \omega_{0,1}^{(j)}}{\partial s_r
\partial s_m} (0)
\omega_{0,1}^{(j)}(0) + 2 \sum_{j} \frac{\partial
\omega_{0,1}^{(j)}}{\partial s_r}(0) \frac{\partial
\omega_{0,1}^{(j)}}{\partial s_m}(0).
\end{align*}

From the above formulae, it follows that
\[
b_{rm}=- \omega_{\mathrm{min}}(B)^{-\frac{2}{k+2}} \Omega_{rm}.
\]

Using integration by parts and (\ref{e:lambda01}), we get
\begin{multline*}
2\int \tau \left(\frac{\tau^{k+1}}{k+1}-\alpha_{\rm min}\right)
\frac{\partial u^0_\alpha}{\partial \tau}(\tau)
u^0_{\alpha_{\mathrm{min}}}(\tau) d\tau \\ + \int \tau
\left(\frac{\tau^{k+1}}{k+1}-\alpha_{\rm min}\right) \frac{\partial
u^0_\alpha}{\partial \tau}(\tau) u^0_{\alpha_{\mathrm{min}}}(\tau)
d\tau=0,
\end{multline*}
that implies that the sum of the second and the third terms in
\eqref{e:C} equals zero. It is also easy to see that the last two
terms in \eqref{e:C} equal zero. Thus, we have
\begin{equation}\label{e:C1}
a =  - \omega_{\mathrm{min}}(B)^{-\frac{2}{k+2}} A.
\end{equation}

We conclude that
\begin{multline}\label{e:RHR}
R_0^+\widehat P_2R_0^- \\ =
-\omega_{\mathrm{min}}(B)^{-\frac{2}{k+2}}
\Delta-\omega_{\mathrm{min}}(B)^{-\frac{2}{k+2}}  \sum_{r,m}
\Omega_{rm} \sigma_r \sigma_m -
\omega_{\mathrm{min}}(B)^{-\frac{2}{k+2}} A.
\end{multline}
Next, by (\ref{e:defE0}) and (\ref{e:defh11}), we have
\begin{multline*}
E_0 \widehat P_1 R_0^-  \\
= 2i \omega_{\mathrm{min}}(B)^{-\frac{1}{k+2}} (I-\Pi_0)
(Q(\alpha_{\rm min},1)-\hat{\nu})^{-1} (I-\Pi_0)
\left(\frac{\tau^{k+1}}{k+1}-\alpha_{min}\right)u^0_{\alpha_{\mathrm{min}}}(\tau)
{\widehat e}_\omega.
\end{multline*}
By (\ref{e:lambda01}), it follows that
\[
\Pi_0 \left[\left( \frac{\tau^{k+1}}{k+1}-\alpha_{\rm min}
\right)u^0_{\alpha_{\mathrm{min}}}(\tau)\right]=\int \left(
\frac{\tau^{k+1}}{k+1}-\alpha_{\rm min}
\right)(u^0_{\alpha_{\mathrm{min}}}(\tau))^2d\tau = 0.
\]
Therefore, by (\ref{e:Qlambda}), we get
\begin{align*}
(I-\Pi_0)\left(\frac{\tau^{k+1}}{k+1}-\alpha_{\rm min}
\right)u^0_{\alpha_{\mathrm{min}}}(\tau)=&
\left(\frac{\tau^{k+1}}{k+1}-\alpha_{\rm min} \right)
u^0_{\alpha_{\mathrm{min}}}(\tau)\\
= & \frac12 \left(Q(\alpha_{\rm min},1)
-\hat{\nu}\right)\frac{\partial u^0_\alpha}{\partial\alpha}.
\end{align*}
We obtain (note that $\Pi_0\frac{\partial
u^0_\alpha}{\partial\alpha}=0$)
\[
E_0 \widehat P_1 R_0^- = i\omega_{\mathrm{min}}(B)^{-\frac{1}{k+2}}
\frac{\partial u^0_\alpha}{\partial\alpha} {\widehat e}_\omega.
\]
Next, we have
\[
R_0^+ \widehat P_1 E_0 \widehat P_1 R_0^-=
2\omega_{\mathrm{min}}(B)^{-\frac{2}{k+2}} R_0^+
\left(\frac{\tau^{k+1}}{k+1}-\alpha_{\rm min}\right)\frac{\partial
u^0_\alpha}{\partial\alpha} \Delta_\omega.
\]
By (\ref{e:lambda02}), it follows that
\begin{multline*}
R_0^+ \left(\frac{\tau^{k+1}}{k+1}-\alpha_{\rm min}
\right)\frac{\partial u^0_\alpha}{\partial\alpha}\\ = \int
\left(\frac{\tau^{k+1}}{k+1}-\alpha_{\rm min}\right)
u^0_{\alpha_{\mathrm{min}}}(\tau) \frac{\partial
u^0_\alpha}{\partial\alpha}(\tau) \,d\tau =  \frac 14 \left( 2-
\frac{\partial^2 \lambda_0}{\partial\alpha^2}(\alpha_{\rm
min},1)\right).
\end{multline*}
Therefore, we obtain
\begin{equation}\label{e:RHEHR}
R_0^+ \widehat P_1 E_0 \widehat P_1 R_0^-= \frac
12\omega_{\mathrm{min}}(B)^{-\frac{2}{k+2}}\left( 2-
\frac{\partial^2 \lambda_0}{\partial\alpha^2}(\alpha_{\rm
min},1)\right) \Delta_\omega.
\end{equation}

From \eqref{def1/k+2}, \eqref{e:RHR} and~\eqref{e:RHEHR}, we get
\eqref{z2}, that completes the proof of the lemma.
\end{proof}

\subsection{Construction of approximate eigenfunctions}
In this section, we complete the proof of Theorem~\ref{YK:l1}. So
suppose that $\lambda$ is in the spectrum of the operator $K$. Our
considerations depend on whether Conjecture~\ref{c:main} is true or
false. Since this is unknown at the moment, we consider both
possible cases.

First, suppose that Conjecture~\ref{c:main} is true, that is, the
second derivative $\frac{\partial^2 \lambda_0}{\partial\alpha^2}
(\alpha_{\mathrm{min}},1)$ is positive. Then the operator $K$ has
discrete spectrum. Let $\phi_0\in \cS(\RR^{n-1})$ be an
eigenfunction of $K$ with the corresponding eigenvalue $\lambda$.
Put in \eqref{e:defZ}
\[
Z_1=0, \quad Z_2=\omega_{\mathrm{min}}(B)^{-\frac{2}{k+2}}\lambda,
\quad Z_\ell=0\ (\forall \ell\geq 3).
\]
So we have
\[
Z(h)=\omega_{\mathrm{min}}(B)^{-\frac{2}{k+2}}\lambda
h^{\frac{1}{k+2}}, \quad
z(h)=\hat{\nu}\omega_{\mathrm{min}}(B)^{\frac{2}{k+2}}
h^{\frac{2k+2}{k+2}}+ \lambda h^{\frac{1}{k+2}}h^{\frac{2k+3}{k+2}}.
\]
By Lemma~\ref{l:expEpm}, we have
\[
{\widehat E}^{\pm}(Z(h)) \phi_0 = O(h^{\frac{3}{2(k+2)}})\;,
\]
and, by \eqref{e:Re1}, we obtain
\[
({\widehat P}^{h}- z(h)) \widehat E^{+}(Z(h))\phi_0
=O(h^{\frac{4k+7}{2(k+2)}}),
\]
where the function $\widehat E^{+}(Z)$ is given by
$$
\widehat E^{+} (Z) = E^{+}(z)\,,
$$
with $Z$ related to $z$ by \eqref{defZ}. From \eqref{e:HhPh}, it
follows that the function
\begin{multline*}
U^h(t,s)=\chi(t,s)h^{-\frac{n+1}{4(k+2)}}
\exp\left(-i\frac{\phi(s)}{h} \right)\exp\left(
i\frac{\alpha_{\mathrm{min}}
\sum\limits_{j=1}^{n-1}\omega_{0,1}^{(j)}(0)s_j}{\omega_{\mathrm{min}}(B)^{\frac{k+1}{k+2}}
h^{\frac{1}{k+2}}}\right)  \\
\times \widehat E^{+}(Z(h))
\phi_0(\omega_{\mathrm{min}}(B)^{\frac{1}{k+2}}
h^{-1/(k+2)}t,h^{-1/2(k+2)}s), \quad s\in B(0,r),\quad t\in \RR,
\end{multline*}
where $\chi\in C^\infty_c(U)$ is a cut-off function, satisfies
$\|U_h\|=1+o(1)$ and
\[
(H^{h}_0- z(h)) U^h =O(h^{\frac{4k+7}{2(k+2)}}).
\]
From the above, we see that $\widehat E^{+}(Z)$ is the following
asymptotic series,
\[
\widehat E^{+}(Z)\sim R_0^-+\sum_{j=1}^{+\infty} (-1)^j E_0[(\delta
P -Z) E_0]^{j-1} [(\delta P -Z) R_0^-],
\]
and, therefore, it admits an asymptotic expansion in powers of
$h^{\frac{1}{2(k+2)}}$:
\[
\widehat E^+(Z(h))\sim \sum_{\ell=1}^\infty
E^{\pm}_{\frac{\ell}{2(k+2)}}h^{\frac{\ell}{2(k+2)}}.
\]
Using this fact, one can easily see that the function $U^h$
satisfies the conditions~\eqref{e:uh}, \eqref{e:u1h} and
\eqref{e:u2h}. By Lemma~\ref{l:model}, it follows that
\[
(H^{h}- z(h)) U^h =O(h^{\frac{4k+7}{2(k+2)}}),
\]
that completes the proof in this case.

Now consider the case when Conjecture~\ref{c:main} is false. So
suppose that the equality $\frac{\partial^2
\lambda_0}{\partial\alpha^2} (\alpha_{\mathrm{min}},1)=0$ is true.
Thus, the operator $K$ has the form
\[
K=\Delta_{\omega^\bot} + \sum_{r,m} \Omega_{rm} \sigma_r \sigma_m
+A.
\]

\begin{lemma}\label{l:quasimodesK}
There exists $w^h_0\in {\mathcal S}(\RR^{n-1})$, $\|w^h_0\|=1$,
which satisfies the following conditions: there exists $C>0$ such
that, for any $h>0$, we have
\[
\|(K-\lambda)w^h_0\| \leq Ch^{\frac{1}{2(k+2)}},
\]
and for any multi-index $\alpha=(\alpha_1,\ldots,\alpha_{n-1})$,
there exists a constant $C_\alpha>0$ such that, for any $h>0$, we
have
\begin{equation}\label{e:wh0}
\left\|\partial^\alpha_\sigma w^h_0\right\| \leq C_\alpha
h^{-\frac{|\alpha|}{2(k+2)}}.
\end{equation}
\end{lemma}

\begin{proof} Take an arbitrary vector $e^\prime_\omega\in \RR^{n-1}$, which
is orthogonal to the vectors $e_2, \ldots, e_{n-1}$ with respect to
the positive definite bilinear form $\Omega$ in $\RR^{n-1}$ given by
\[
\Omega(\sigma,\sigma^\prime)=\sum_{r,m} \Omega_{rm} \sigma_r
\sigma^\prime_m.
\]
Consider the linear coordinate system in $\RR^{n-1}$ with
coordinates $(\rho_1,\rho_2,\ldots,\rho_{n-1})$ defined by the base
$(e^\prime_\omega, e_2, \ldots, e_{n-1})$. In these coordinates, the
operator $K$ has the form
\[
K=-\sum_{j=2}^{n-1}\frac{\partial^2}{\partial
\rho^2_j}+\Omega^\prime_{11}\rho_1^2+\sum_{r,m=2}^{n-2}
\Omega^\prime_{rm} \rho_r \rho_m,
\]
where $\Omega^\prime_{11}>0$ and the quadratic form
$\sum_{r,m=2}^{n-2} \Omega^\prime_{rm} \rho_r \rho_m$ is positive
definite.

With respect to the decomposition $L^2(\RR^{n-1})=L^2(\RR_{\rho_1})
\otimes L^2(\RR^{n-2}_{(\rho_2,\ldots,\rho_{n-1})})$, the operator
$K$ can be written as $K=V\otimes I+I\otimes K_0$, where $V$ is the
multiplication operator in $L^2(\RR_{\rho_1})$ by the function
\[
V(\rho_1)=\Omega^\prime_{11}\rho_1^2,\quad \rho_1\in \RR,
\]
and $K_0$ is the harmonic oscillator in $L^2(\RR^{n-2}_{\rho_2,
\ldots,\rho_{n-1}})$ given by
\[
K_0=-\sum_{j=2}^{n-1}\frac{\partial^2}{\partial
\rho^2_j}+\sum_{r,m=2}^{n-2} \Omega^\prime_{rm} \rho_r \rho_m.
\]
Therefore (see, for instance, \cite[Theorem VIII.33]{RSI}), the
spectrum of $K$ equals
\[
\sigma(K)=\overline{\sigma(V)+\sigma(K_0)}=[\lambda_0(K_0),
+\infty),
\]
where $\lambda_0(K_0)=\inf \sigma(K_0)$.

Denote by $u_0\in {\mathcal S}(\RR^{n-2})$, $\|u_0\|=1$, the
eigenfunction of $K_0$ associated with $\lambda_0(K_0)$. Define the
function $w^h_0\in {\mathcal S}(\RR^{n-1})$ by the formula
\begin{equation}\label{e:wh0-explicit}
w^h_0(\rho)=ch^{-1/2(k+2)}e^{-(\rho_1-a)^2/(2h^{2/(k+2)})}
u_0(\rho_2,\ldots,\rho_{n-1}),\quad \rho\in\RR^{n-1},
\end{equation}
where the constant $c>0$ is chosen in such a way that $\|w^h_0\|=1$
and $a$ satisfies the condition $$\lambda =\lambda_0(K_0)
+\Omega^\prime_{11}a^2\,.
$$
 Then
\begin{multline*}
Kw^h_0(\rho)= \lambda_0(K_0)w^h_0(\rho) \\ +
ch^{-1/2(k+2)}\Omega^\prime_{11}\rho_1^2
e^{-(\rho_1-a)^2/(2h^{2/(k+2)})} u_0(\rho_2,\ldots,\rho_{n-1}),\quad
\rho\in\RR^{n-1},
\end{multline*}
and, therefore, we have
\begin{align*}
\|Kw^h_0-\lambda w^h_0\| & = c\Omega^\prime_{11} h^{-1/2(k+2)}
\left(\int (\rho_1^2 - a^2)^2 e^{-(\rho_1-a)^2/h^{2/(k+2)}} d\rho_1
\right)^{1/2}\\ & \leq Ch^{1/2(k+2)}.
\end{align*}
The estimates \eqref{e:wh0} follow immediately from the explicit
formula \eqref{e:wh0-explicit} for $w^h_0$.
\end{proof}

Take $w^h_0$ as in Lemma~\ref{l:quasimodesK} and
\[
 w^h_1=-(\widehat P_0-\hat{\nu})^{-1}(I-\Pi_0)\Big(\widehat P_2R_0^- -\widehat P_1E_0
\widehat P_1 R_0^- \Big) w^h_0.
\]
The operator $-(\widehat P_0-\hat{\nu})^{-1}(I-\Pi_0)\Big(\widehat
P_2 -\widehat P_1E_0 \widehat P_1 \Big)$ has the form of a second
order differential operator in $\sigma$, whose coefficients are
bounded operators in $L^2(\RR,d\tau)$. Using this fact, it can be
easily checked that, for any multi-index
$\alpha=(\alpha_1,\ldots,\alpha_{n-1})$, there exists a constant
$C_\alpha>0$ such that, for any $h>0$, we have
\begin{equation}\label{e:wh1}
\left\|\partial^\alpha_\sigma w^h_1\right\| \leq C_\alpha
h^{-\frac{|\alpha|}{2(k+2)}}.
\end{equation}

Put
\begin{align*}
v^h(\tau,\sigma)= & (R_0^--h^{\frac{1}{2(k+2)}}E_0 \widehat P_1 R_0^-)
w^h_0(\tau,\sigma)+h^{\frac{1}{k+2}} R_0^-w^h_1(\tau,\sigma)\\
= & u^0_{\alpha_{\mathrm{min}}}(\tau)w^h_0(\sigma) - i
h^{\frac{1}{2(k+2)}} \omega_{\mathrm{min}}(B)^{-\frac{1}{k+2}}
\frac{\partial u^0_\alpha}{\partial\alpha}(\tau) {\widehat e}_\omega
w^h_0(\sigma)\\
& +h^{\frac{1}{k+2}}w^h_1(\tau,\sigma).
\end{align*}
Then, with $z(h)=\hat{\nu}\omega_{\mathrm{min}}(B)^{\frac{2}{k+2}}
h^{\frac{2k+2}{k+2}}+\lambda h^{\frac{2k+3}{k+2}}$, we have
\begin{align*}
({\widehat P}^{h}-z(h))v^h =&
\omega_{\mathrm{min}}(B)^{\frac{2}{k+2}}
h^{\frac{2k+2}{k+2}}\Big[(\widehat P_0-\hat{\nu})R_0^-w^h_0\\
& +h^{\frac{1}{2(k+2)}} \Big(\widehat P_1 R_0^-w^h_0-(\widehat
P_0-\hat{\nu})E_0
\widehat P_1 R_0^-w^h_0 \Big)\\
& + h^{\frac{1}{k+2}} \Big(\Big((\widehat P_2-
\omega_{\mathrm{min}}(B)^{-\frac{2}{k+2}}\lambda)R_0^- -\widehat
P_1E_0 \widehat P_1 R_0^- \Big) w^h_0 \\ & +(\widehat P_0-\hat{\nu})
w^h_1\Big)\Big] + O(h^{\frac{4k+7}{2(k+2)}})\\
= & \omega_{\mathrm{min}}(B)^{\frac{2}{k+2}}
h^{\frac{2k+3}{k+2}}\Big[\Pi_0 \Big((\widehat P_2 -
\omega_{\mathrm{min}}(B)^{-\frac{2}{k+2}}\lambda)R_0^-\\ & -\widehat
P_1E_0
\widehat P_1 R_0^- \Big) w^h_0\Big] + O(h^{\frac{4k+7}{2(k+2)}})\\
= & h^{\frac{2k+3}{k+2}}R_0^- (K-\lambda) w^h_0+
O(h^{\frac{4k+7}{2(k+2)}}),
\end{align*}
Therefore, we obtain
\[
\|({\widehat P}^{h}-z(h))v^h\|=O(h^{\frac{4k+7}{2(k+2)}}).
\]
From \eqref{e:HhPh}, it follows that the function
\begin{multline*}
U^h(t,s)=\chi(t,s) h^{-\frac{n+1}{4(k+2)}}
\exp\left(-i\frac{\phi(s)}{h} \right)\exp\left(
i\frac{\alpha_{\mathrm{min}}
\sum\limits_{j=1}^{n-1}\omega_{0,1}^{(j)}(0)s_j}{\omega_{\mathrm{min}}(B)^{\frac{k+1}{k+2}}
h^{\frac{1}{k+2}}}\right)  \\
\times v^h(\omega_{\mathrm{min}}(B)^{\frac{1}{k+2}}
h^{-1/(k+2)}t,h^{-1/2(k+2)}s), \quad s\in B(0,r),\quad t\in \RR,
\end{multline*}
where $\chi$ is a cut-off function, satisfies $\|U_h\|=1+o(1)$ and
\[
(H^{h}_0- z(h)) U^h =O(h^{\frac{4k+7}{2(k+2)}}).
\]
Using \eqref{e:wh0} and \eqref{e:wh1}, one can easily verify that
the function $U^h$ satisfies the conditions~\eqref{e:uh},
\eqref{e:u1h} and \eqref{e:u2h}. By Lemma~\ref{l:model}, it follows
that
\[
(H^{h}- z(h)) U^h =O(h^{\frac{4k+7}{2(k+2)}})\,,
\]
that completes the proof of Theorem~\ref{YK:l1}.
\medskip\par
If Conjecture~\ref{c:main} is true, following the arguments of
\cite{Fournais-Helffer06}, one can prove the following refined
version of Theorem~\ref{YK:l1} (cf. \cite[Theorem
3.1]{Fournais-Helffer06}): for any $\lambda$ in the spectrum of $K$,
there exist a sequence $\{\zeta_j\}_{j=0}^\infty\subset \RR$ and a
sequence of functions $\{\phi_j\}_{j=0}^\infty$ in $C^\infty_c(D)$
such that, for any $N>0$, there exists $M>0$ such that, if
\[
z_M(h)= \hat{\nu} \omega_{\mathrm{min}}(B)^{\frac{2}{k+2}}
h^{\frac{2k+2}{k+2}}+ \lambda h^{\frac{2k+3}{k+2}}+
h^{\frac{4k+7}{2(k+2)}}\sum_{j=0}^M h^{\frac{j}{2(k+2)}}\zeta_j\,.
\]
and
\[
\phi^h_M(x)=\sum_{j=0}^M h^{\frac{j}{2(k+2)}}\phi_j(x)
\]
then
\[
\|\left(H^{h}-z_M(h)\right)\phi^h_M\|_{L^2}\leq
O(h^{N})\|\phi^h_M\|_{L^2}, \quad h\to 0.
\]
Moreover, if, in addition, Conjecture~\ref{c:main2} is true and a
miniwell $x_1\in S$ is unique, one can show, following the lines of
\cite{Fournais-Helffer06}, that the $m$-th eigenvalue
$\lambda_m(H^h)$ of the operator $H^h$ admits the asymptotic
expansion
\[
\lambda_m(H^h) = \hat{\nu} \omega_{\mathrm{min}}(B)^{\frac{2}{k+2}}
h^{\frac{2k+2}{k+2}}+ \lambda_m h^{\frac{2k+3}{k+2}}+
h^{\frac{4k+7}{2(k+2)}}\sum_{j=0}^\infty
h^{\frac{j}{2(k+2)}}\zeta_j\,,
\]
where $\lambda_m$ is $m$-th eigenvalue of the operator $K$.

\subsection{Proof of Theorem~\ref{t:gaps}}
To complete the proof of Theorem~\ref{t:gaps}, we will use a general
result on the existence of gaps in the spectrum of the magnetic
Schr\"odinger operator $H^h$ on the interval $[0,
h(b_0+\epsilon_0)]$ obtained in~\cite[Theorem 2.1]{diff2006}. Fix
$\epsilon_1>0$ and $\epsilon_2>0$ such that $\epsilon_1 < \epsilon_2
< \epsilon_0$, and consider the Dirichlet realization $H^h_D$ on the
operator $H^h$ in the domain $D=\overline{U_{\epsilon_2}}$. The
operator $H^h_D$ has discrete spectrum.

\begin{theorem} \label{t:abstract}
Let $N\geq 1$. Suppose that there is a subset
$\mu_0^h<\mu_1^h<\ldots <\mu_{N}^h$ of an interval $I(h)\subset [0,
h(b_0+\epsilon_1))$ such that
\begin{enumerate}
  \item There exist constants $c>0$ and $M \geq 1$ such that
\[
\begin{split}
& \mu_{j}^h-\mu_{j-1}^h>ch^M, \quad j=1,\ldots,N,\\
& {\rm dist}(\mu_0^h,\partial I(h))>ch^M,\quad {\rm
dist}(\mu_N^h,\partial I(h))>ch^M,
\end{split}
\]
for any $h>0$ small enough;
  \item Each $\mu_j^h, j=0,1,\ldots,N,$ is an approximate eigenvalue of
the operator $H^h_D$: for some $v_j^h\in C^\infty_c(D)$ we have
\[
\|H^h_Dv_j^h-\mu^h_jv_j^h\|=\alpha_j(h)\|v_j^h\|,
\]
where $\alpha_j(h)=o(h^M)$ as $h\to 0$.
\end{enumerate}
Then the spectrum of $H^h$ on the interval $I(h)$ has at least $N$
gaps for any sufficiently small $h>0$.
\end{theorem}

Fix any $N\geq 1$. If $\frac{\partial^2 \lambda_0}{\partial\alpha^2}
(\alpha_{\mathrm{min}},1) >0$, then $K$ has discrete spectrum
$\lambda_0 < \lambda_1<\ldots$, $\lambda_j\to \infty$ as $j\to
\infty$. Take an arbitrary $b_N > \lambda_N$. By
Theorem~\ref{YK:l1}, for any $j=0,1,\ldots,N$, there exist $C_j>0$,
$h_{0,j}>0$, $\phi_j(h)\in C^\infty_c(D)$ and $z_j(h)$ with
\[
z_j(h)= \hat{\nu} \omega_{\mathrm{min}}(B)^{\frac{2}{k+2}}
h^{\frac{2k+2}{k+2}}+ \lambda_j h^{\frac{2k+3}{k+2}}+
O(h^{\frac{4k+7}{2k+4}})
\]
such that, for any $h\in (0,h_{0,j}]$, we have
\[
\|\left(H^{h}_D-z_j(h)\right)\phi(h)\|\leq C_j
h^{\frac{4k+7}{2k+4}}\|\phi(h)\|.
\]
So we apply Theorem~\ref{t:abstract} with
\[
I(h)= \left[\hat{\nu}\,
\omega_{\mathrm{min}}(B)^{\frac{2}{k+2}}h^{\frac{2k+2}{k+2}},
\hat{\nu}\,
\omega_{\mathrm{min}}(B)^{\frac{2}{k+2}}h^{\frac{2k+2}{k+2}} +
b_Nh^{\frac{2k+3}{k+2}}\right]
\]
and $\mu_j^h=z_j(h), j=0,1,\ldots,N$, that completes the proof in
this case.

If $\frac{\partial^2 \lambda_0}{\partial\alpha^2}
(\alpha_{\mathrm{min}},1)=0$, then the spectrum of $K$ is a
semi-axis $[\lambda_0,\infty)$. Taking an arbitrary $b_N >
\lambda_0$ and arbitrary $\lambda_1<\ldots < \lambda_N$ on the
interval $(\lambda_0,b_N)$ and proceeding as above, we complete the
proof.

\end{document}